\documentclass[11pt]{article}
\usepackage[]{amsmath,amssymb}
\usepackage{cite}
\newtheorem{theorem}{Theorem}[section]
\newtheorem{lemma}[theorem]{Lemma}
\newtheorem{proposition}[theorem]{Proposition}
\newtheorem{corollary}[theorem]{Corollary}
\newtheorem{exAux}[theorem]{Example}
\newenvironment{example}{\begin{exAux} \rm}{\end{exAux}}
\newtheorem{Def}[theorem]{Definition}
\newenvironment{definition}{\begin{Def} \rm}{\end{Def}}
\newtheorem{Note}[theorem]{Note}
\newenvironment{note}{\begin{Note} \rm}{\end{Note}}
\newtheorem{Problem}[theorem]{Problem}
\newenvironment{problem}{\begin{Problem} \rm}{\end{Problem}}
\newtheorem{Rem}[theorem]{Remark}

\newtheorem{Not}[theorem]{Notation}

\newtheorem{Conj}[theorem]{Conjecture}
\newenvironment{conjecture}{\begin{Conj}}{\end{Conj}}
\newtheorem{Ass}[theorem]{Assumption}
\newenvironment{assumption}{\begin{Ass} \rm}{\end{Ass}}

\newenvironment{proof}{\medskip\noindent{\bf Proof.\ }}{\qed\medskip}

\newcommand{\qed}{\hfill\mbox{$\Box$\qquad\qquad}}

\newcommand{\ve}{\varepsilon}

\renewcommand{\b}[1]{\langle #1 \rangle}
\newcommand{\vphi}{y}
\renewcommand{\th}{\theta}
\newcommand{\be}{\beta}

\newcommand{\K}{\mathbb{F}}
\renewcommand{\indent}{\hspace{6mm}}

\begin{document}
\thispagestyle{empty}

\begin{center}
\LARGE \bf
Tridiagonal pairs and
the $\mu$-conjecture
\end{center}

\smallskip

\begin{center}
\Large
Kazumasa Nomura and 
Paul Terwilliger\footnote{This author gratefully acknowledges 
support from the FY2007 JSPS Invitation Fellowship Program
for Reseach in Japan (Long-Term), grant L-07512.}
\end{center}

\smallskip

\begin{quote}
\small 
\begin{center}
\bf Abstract
\end{center}
\indent
Let $\K$ denote a field and let $V$ denote a vector space over $\K$ with 
finite positive dimension.
We consider a pair of linear transformations $A:V \to V$
and $A^*:V \to V$ that satisfy the following conditions:
(i)
each of $A,A^*$ is diagonalizable;
(ii)
there exists an ordering $\lbrace V_i\rbrace_{i=0}^d$ of the eigenspaces of 
$A$ such that
$A^* V_i \subseteq V_{i-1} + V_{i} + V_{i+1}$ for $0 \leq i \leq d$,
where $V_{-1}=0$ and $V_{d+1}=0$;
(iii)
there exists an ordering $\lbrace V^*_i\rbrace_{i=0}^\delta$ 
of the eigenspaces of $A^*$ such that
$A V^*_i \subseteq V^*_{i-1} + V^*_{i} + V^*_{i+1}$ for
 $0 \leq i \leq \delta$,
where $V^*_{-1}=0$ and $V^*_{\delta+1}=0$;
(iv) 
there is no subspace $W$ of $V$ such that
$AW \subseteq W$, $A^* W \subseteq W$, $W \neq 0$, $W \neq V$.
We call such a pair a {\it tridiagonal pair} on $V$.
It is known that $d=\delta$ and for $0 \leq i \leq d$
the dimensions of $V_i$, $V_{d-i}$, $V^*_i$, $V^*_{d-i}$ coincide.
We say the pair $A,A^*$ is {\it sharp} whenever $\dim V_0=1$.
It is known that if 
$\K$ is algebraically closed then $A,A^*$ is sharp.
A conjectured classification of the sharp tridiagonal pairs
was recently introduced by T. Ito and the second author.
We present a result which supports the conjecture.
Given scalars $\{\th_i\}_{i=0}^d$, $\{\th^*_i\}_{i=0}^d$ in $\K$ that 
satisfy the known constraints on the eigenvalues of a tridiagonal pair,
we define an $\K$-algebra $T$ by generators and relations.
We consider the $\K$-algebra $e^*_0Te^*_0$ for a certain idempotent 
$e^*_0 \in T$.
Let $\K[x_1,\ldots,x_d]$ denote the polynomial algebra over $\K$ involving 
$d$ mutually commuting indeterminates.
We display a surjective $\K$-algebra homomorphism
$\mu: \K[x_1,\ldots,x_d] \to e^*_0Te^*_0$.
We conjecture that $\mu$ is an isomorphism.
We show that this $\mu$-conjecture implies the classification conjecture, 
and that  the $\mu$-conjecture holds for $d\leq 5$.


\end{quote}

\section{Tridiagonal pairs}

\indent
Throughout this paper $\K$ denotes a field.

\medskip

We begin by recalling the notion of a tridiagonal pair. 
We will use the following terms.
Let $V$ denote a vector space over $\K$ with finite positive dimension.
For a linear transformation $A:V\to V$ and a subspace $W \subseteq V$,
we call $W$ an {\em eigenspace} of $A$ whenever $W\not=0$ and there exists 
$\th \in \K$ such that $W=\{v \in V \,|\, Av = \th v\}$;
in this case $\th$ is the {\em eigenvalue} of $A$ associated with $W$.
We say that $A$ is {\em diagonalizable} whenever $V$ is spanned by the 
eigenspaces of $A$.

\begin{definition}  {\rm \cite[Definition 1.1]{TD00}} \label{def:tdp}  \samepage
Let $V$ denote a vector space over $\K$ with finite positive dimension. 
By a {\em tridiagonal pair} on $V$ we mean an ordered pair of linear 
transformations $A:V \to V$ and $A^*:V \to V$ that satisfy the following 
four conditions.
\begin{enumerate}
\item 
Each of $A,A^*$ is diagonalizable.
\item 
There exists an ordering $\{V_i\}_{i=0}^d$ of the eigenspaces of $A$ 
such that 
\begin{equation}               \label{eq:t1}
A^* V_i \subseteq V_{i-1} + V_i+ V_{i+1} \qquad \qquad (0 \leq i \leq d),
\end{equation}
where $V_{-1} = 0$ and $V_{d+1}= 0$.
\item
There exists an ordering $\{V^*_i\}_{i=0}^{\delta}$ of the eigenspaces of 
$A^*$ such that 
\begin{equation}                \label{eq:t2}
A V^*_i \subseteq V^*_{i-1} + V^*_i+ V^*_{i+1} 
\qquad \qquad (0 \leq i \leq \delta),
\end{equation}
where $V^*_{-1} = 0$ and $V^*_{\delta+1}= 0$.
\item 
There does not exist a subspace $W$ of $V$ such  that $AW\subseteq W$,
$A^*W\subseteq W$, $W\not=0$, $W\not=V$.
\end{enumerate}
We say the pair $A,A^*$ is {\it over $\K$}.
We call $V$ the {\it vector space underlying $A,A^*$}.
\end{definition}

\begin{note}   \label{note:star}        \samepage
According to a common notational convention $A^*$ denotes 
the conjugate-transpose of $A$. We are not using this convention.
In a tridiagonal pair $A,A^*$ the linear transformations $A$ and $A^*$
are arbitrary subject to (i)--(iv) above.
\end{note}

\medskip
 
We now summarize what is known about tridiagonal pairs.
Let $A,A^*$ denote a tridiagonal pair on $V$, as in Definition \ref{def:tdp}. 
By \cite[Lemma 4.5]{TD00} the integers $d$ and $\delta$ from (ii), (iii) are 
equal; we call this common value the {\em diameter} of the pair.
By \cite[Theorem 10.1]{TD00} the pair $A, A^*$ satisfy two polynomial
equations called the tridiagonal relations;
these generalize the $q$-Serre relations   \cite[Example~3.6]{qSerre}
and the Dolan-Grady relations   \cite[Example~3.2]{qSerre}.
See \cite{bas1,bas2,bas3,bas4,bas5,bas6,bas7,N:aw,qSerre,aw}
for results on the tridiagonal relations.
An ordering of the eigenspaces of $A$ (resp. $A^*$)
is said to be {\em standard} whenever it satisfies
\eqref{eq:t1} (resp. \eqref{eq:t2}). We comment on the uniqueness of the
standard ordering. Let $\{V_i\}_{i=0}^d$ denote a standard
ordering of the eigenspaces of $A$. By \cite[Lemma 2.4]{TD00}, the ordering
$\{V_{d-i}\}_{i=0}^d$ is also standard and no further ordering
is standard. A similar result holds for the eigenspaces
of $A^*$. Let $\{V_i\}_{i=0}^d$ (resp. $\{V^*_i\}_{i=0}^d$) 
denote a standard ordering of the eigenspaces of $A$ (resp. $A^*$).
By \cite[Corollary 5.7]{TD00}, for $0 \leq i \leq d$ the spaces $V_i$, $V^*_i$
have the same dimension; we denote this common dimension by $\rho_i$. 
By \cite[Corollaries 5.7, 6.6]{TD00} the sequence $\{\rho_i\}_{i=0}^d$ is 
symmetric and unimodal;
that is $\rho_i=\rho_{d-i}$ for $0 \leq i \leq d$ and
$\rho_{i-1} \leq \rho_i$ for $1 \leq i \leq d/2$.
We call the sequence $\{\rho_i\}_{i=0}^d$ the {\em shape} of $A,A^*$.
See \cite{shape, NN, IT:Krawt, IT:aug, nom4, N:refine, nomsplit}
for results on the shape.  
We say $A,A^*$ is {\it sharp} whenever $\rho_0=1$.
By \cite[Theorem~1.3]{nomstructure}, if $\K$ is algebraically closed then
$A,A^*$ is sharp.
By \cite[Theorem~1.4]{nomstructure}, if $A,A^*$ is sharp then there exists a 
nondegenerate symmetric bilinear form $\b{\,,\,}$ on $V$ such that
$\b{Au,v}=\b{u,Av}$ and $\b{A^*u,v}=\b{u,A^*v}$ for all $u,v \in V$.
See \cite{CurtH,nomsharp} for results on the bilinear form.
The following special cases of tridiagonal pairs have been studied extensively.
In \cite{Vidar} the tridiagonal pairs of shape $(1,2,1)$ are classified and 
described in detail.
The tridiagonal pairs of shape $(1,1,\ldots,1)$ are called {\em Leonard pairs}
\cite[Definition 1.1]{LS99}, and these are classified in \cite{LS99,TLT:array}.
This classification yields a correspondence between the Leonard pairs and a
family of orthogonal polynomials consisting of the $q$-Racah polynomials
and their relatives \cite{AWil,qrac}.
This family coincides with the terminating branch of the Askey scheme 
\cite{KoeSwa}.
See \cite{NT:balanced,NT:formula,NT:det,NT:mu,NT:span,NT:switch,madrid}
and the references therein for results on Leonard pairs.
For the above tridiagonal pair $A,A^*$ and for $0 \leq i \leq d$ let $\th_i$ 
(resp. $\th^*_i$) denote the eigenvalue of $A$ (resp. $A^*$) associated with 
$V_i$ (resp. $V^*_i$).
The pair $A,A^*$ is said to have {\em Krawtchouk type}
(resp. {\em $q$-geometric type})
whenever $\th_i=d-2i$ (resp. $\th_i=q^{d-2i}$)
and $\th^*_i=d-2i$ (resp. $\th^*_i=q^{d-2i}$) for $0 \leq i \leq d$.
In \cite[Theorems 1.7, 1.8]{Ha} the tridiagonal pairs of Krawtchouk type are 
classified.
By \cite[Remark~1.9]{Ha} these tridiagonal pairs  are in bijection with the 
finite-dimensional irreducible modules for the  three-point loop algebra
$\mathfrak{sl}_2 \otimes \K[t,t^{-1}, (t-1)^{-1}]$.
See \cite{Ev, IT:Krawt} for results on  tridiagonal pairs of Krawtchouk type.
In \cite[Theorems~1.6,~1.7]{NN} the $q$-geometric tridiagonal pairs are
classified.
By \cite[Theorems~10.3,~10.4]{qtet} these tridiagonal pairs are in bijection 
with the type $1$, finite-dimensional, irreducible modules for the algebra 
$\boxtimes_q$; this is a $q$-deformation of 
$\mathfrak{sl}_2 \otimes \K[t,t^{-1}, (t-1)^{-1}]$ as explained in \cite{qtet}.
See \cite{hasan,hasan2,shape,tdanduq,NN,ITdrg} for results  on  $q$-geometric
tridiagonal pairs.

\medskip

We now summarize the present paper.
A conjectured classification of the sharp tridiagonal pairs was introduced 
in \cite[Conjecture~14.6]{IT:Krawt} and studied carefully in \cite{nomsharp}; 
see  Conjecture \ref{conj:main} below. 
In the present paper we obtain two results which clarify the conjecture and 
provide some more evidence that it is true.
To describe these results, we start with a sequence of scalars 
$(\{\th_i\}_{i=0}^d$, $\{\th^*_i\}_{i=0}^d)$
taken from $\K$ that satisfy the known constraints on the eigenvalues of a 
tridiagonal pair over $\K$;
these are conditions (i) and (iii) in Conjecture \ref{conj:main}.
We associate with this sequence an $\K$-algebra $T$ defined
by generators and relations; $T$ is reminiscent of
an algebra introduced by E. Egge \cite[Definition 4.1]{Egge}.
We are interested in the $\K$-algebra $e^*_0Te^*_0$ where $e^*_0$ is a certain 
idempotent element of $T$.
Let $\lbrace x_i\rbrace_{i=1}^d$ denote mutually commuting indeterminates.
Let $\K[x_1,\ldots,x_d]$ denote the $\K$-algebra consisting of the polynomials
in $\{x_i\}_{i=1}^d$ that have all coefficients in $\K$.
We display a surjective $\K$-algebra homomorphism
$\mu: \K[x_1,\ldots,x_d] \to e^*_0Te^*_0$.
We conjecture that $\mu$ is an isomorphism; let us call this
the {\em $\mu$-conjecture}.
Our two main results are that the $\mu$-conjecture
implies the classification conjecture, 
and that the $\mu$-conjecture holds for $d\leq 5$.
These results are contained in Theorems \ref{thm:main} and \ref{thm:mainpd5}.

\section{Tridiagonal systems}

\indent
When working with a tridiagonal pair, it is often convenient to consider
a closely related object called a tridiagonal system.
To define a tridiagonal system, we recall a few concepts from linear algebra.
Let $V$ denote a vector space over $\K$ with finite positive dimension.
Let ${\rm End}(V)$ denote the $\K$-algebra of all linear
transformations from $V$ to $V$.
Let $A$ denote a diagonalizable element of $\text{End}(V)$.
Let $\{V_i\}_{i=0}^d$ denote an ordering of the eigenspaces of $A$
and let $\{\th_i\}_{i=0}^d$ denote the corresponding ordering of the 
eigenvalues of $A$.
For $0 \leq i \leq d$ define $E_i \in \text{End}(V)$ such that 
$(E_i-I)V_i=0$ and $E_iV_j=0$ for $j \neq i$ $(0 \leq j \leq d)$.
Here $I$ denotes the identity of $\mbox{\rm End}(V)$.
We call $E_i$ the {\em primitive idempotent} of $A$ corresponding to $V_i$
(or $\th_i$).
Observe that
(i) $\sum_{i=0}^d E_i = I$;
(ii) $E_iE_j=\delta_{i,j}E_i$ $(0 \leq i,j \leq d)$;
(iii) $V_i=E_iV$ $(0 \leq i \leq d)$;
(iv) $A=\sum_{i=0}^d \theta_i E_i$.
Moreover
\begin{equation}         \label{eq:defEi}
  E_i=\prod_{\stackrel{0 \leq j \leq d}{j \neq i}}
          \frac{A-\theta_jI}{\theta_i-\theta_j}.
\end{equation}
Note that each of $\{A^i\}_{i=0}^d$, $\{E_i\}_{i=0}^d$ is a basis for the 
$\K$-subalgebra of $\mbox{\rm End}(V)$ generated by $A$.
Moreover $\prod_{i=0}^d(A-\theta_iI)=0$.
Now let $A,A^*$ denote a tridiagonal pair on $V$.
An ordering of the primitive idempotents or eigenvalues of $A$ (resp. $A^*$)
is said to be {\em standard} whenever the corresponding ordering of the 
eigenspaces of $A$ (resp. $A^*$) is standard.

\medskip

\begin{definition} \cite[Definition 2.1]{TD00} \label{def:TDsystem}  \samepage
Let $V$ denote a vector space over $\K$ with finite positive dimension.
By a {\em tridiagonal system} on $V$ we mean a sequence
\[
 \Phi=(A;\{E_i\}_{i=0}^d;A^*;\{E^*_i\}_{i=0}^d)
\]
that satisfies (i)--(iii) below.
\begin{itemize}
\item[(i)]
$A,A^*$ is a tridiagonal pair on $V$.
\item[(ii)]
$\{E_i\}_{i=0}^d$ is a standard ordering
of the primitive idempotents of $A$.
\item[(iii)]
$\{E^*_i\}_{i=0}^d$ is a standard ordering
of the primitive idempotents of $A^*$.
\end{itemize}
We say $\Phi$ is {\em over} $\K$.
We call $V$ the {\it vector space underlying $\Phi$}.
\end{definition}

\medskip

The notion of isomorphism for tridiagonal systems
is defined in \cite[Definition 3.1]{nomsharp}.

\medskip

The following result is immediate from lines (\ref{eq:t1}),  (\ref{eq:t2})
and Definition \ref{def:TDsystem}.

\medskip

\begin{lemma} {\rm \cite[Lemma 2.5]{nomtowards}}  \label{lem:trid}  \samepage
Let $(A;\{E_i\}_{i=0}^d;A^*;\{E^*_i\}_{i=0}^d)$ denote a tridiagonal system.
Then for $0 \leq i,j,k \leq d$ the following {\rm (i)}, {\rm (ii)} hold.
\begin{itemize}
\item[\rm (i)]
$E^*_i A^k E^*_j =0\;\;$ if $k<|i-j|$.
\item[\rm (ii)]
$E_i A^{*k} E_j = 0\;\;$ if  $k<|i-j|$.
\end{itemize}
\end{lemma}

\begin{definition}        \label{def}        \samepage
Let $\Phi=(A;\{E_i\}_{i=0}^d;A^*;\{E^*_i\}_{i=0}^d)$ denote a tridiagonal 
system on $V$.
For $0 \leq i \leq d$ let $\th_i$ (resp. $\th^*_i$) denote the eigenvalue of 
$A$ (resp. $A^*$) associated with the eigenspace $E_iV$ (resp. $E^*_iV$).
We call $\{\th_i\}_{i=0}^d$ (resp. $\{\th^*_i\}_{i=0}^d$)
the {\em eigenvalue sequence} (resp. {\em dual eigenvalue sequence}) of $\Phi$.
We observe that $\{\th_i\}_{i=0}^d$ (resp. $\{\th^*_i\}_{i=0}^d$) are mutually 
distinct and contained in $\K$. 
We say $\Phi$ is {\it sharp} whenever the tridiagonal pair $A,A^*$ is sharp.
\end{definition}

\medskip

We now recall the split sequence of a tridiagonal system.
We will use the following notation.

\medskip

\begin{definition}  \label{def:tau}    \samepage
Let $\lambda$ denote an indeterminate and let $\K[\lambda]$ denote the 
$\K$-algebra consisting of the polynomials in $\lambda$ that have 
all coefficients in $\K$.
Let $d$ denote a nonnegative integer and let
$(\{\th_i\}_{i=0}^d; \{\th^*_i\}_{i=0}^d)$ denote a sequence of scalars 
taken from $\K$.
Then for $0 \leq i \leq d$ we define the following polynomials in $\K[\lambda]$:
\begin{align*}
 \tau_i &= (\lambda-\th_0)(\lambda-\th_1)\cdots(\lambda -\th_{i-1}), \\
 \eta_i &= (\lambda-\th_d)(\lambda-\th_{d-1})\cdots(\lambda-\th_{d-i+1}),  \\
 \tau^*_i &= (\lambda-\th^*_0)(\lambda-\th^*_1)\cdots(\lambda-\th^*_{i-1}), \\
 \eta^*_i &= (\lambda-\th^*_d)(\lambda-\th^*_{d-1})\cdots(\lambda-\th^*_{d-i+1}).
\end{align*}
Note that each of $\tau_i$, $\eta_i$, $\tau^*_i$, $\eta^*_i$ is monic with 
degree $i$.
\end{definition}

\medskip

The following definition of the split sequence is
motivated by \cite[Lemma 5.4]{nomstructure}.

\medskip

\begin{definition}  \label{def:split} \samepage
Let $(A; \{E_i\}_{i=0}^d; A^*; \{E^*_i\}_{i=0}^d)$ denote a sharp tridiagonal 
system over $\K$, with eigenvalue sequence $\{\th_i\}_{i=0}^d$
and dual eigenvalue sequence $\{\th^*_i\}_{i=0}^d$.
By \cite[Lemma 5.4]{nomstructure},
for $0 \leq i \leq d$ there exists a unique $\zeta_i \in \K$ such that 
\[
E^*_0 \tau_i(A) E^*_0 = 
\frac{\zeta_i E^*_0}
{(\theta^*_0-\theta^*_1)(\theta^*_0-\theta^*_2)\cdots(\theta^*_0-\theta^*_i)}. 
\]
We note that $\zeta_0=1$.
We call $\lbrace \zeta_i \rbrace_{i=0}^d$ the {\em split sequence} of the 
tridiagonal system.
\end{definition}

\begin{definition} \cite[Definition 6.2]{nomsharp}   \samepage
Let $\Phi$ denote a sharp tridiagonal system.
By the {\em parameter array} of $\Phi$ we mean the sequence
 $(\{\theta_i\}_{i=0}^d; \{\theta^*_i\}_{i=0}^d; \{\zeta_i\}_{i=0}^d)$
where  $\{\theta_i\}_{i=0}^d$ (resp. $\{\theta^*_i\}_{i=0}^d$) is the 
eigenvalue sequence (resp. dual eigenvalue sequence) of $\Phi$ and
$\{\zeta_i\}_{i=0}^d$ is the split sequence of $\Phi$.
\end{definition}

\section{The classification conjecture}

\indent
In this section we discuss a conjectured classification of the sharp 
tridiagonal systems. 

\begin{conjecture} {\rm \cite[Conjecture~14.6]{IT:Krawt}}\label{conj:main}
\samepage  
Let $d$ denote a nonnegative integer and let
\begin{equation}         \label{eq:parray}
 (\{\th_i\}_{i=0}^d; \{\th^*_i\}_{i=0}^d; \{\zeta_i\}_{i=0}^d)
\end{equation}
denote a sequence of scalars taken from $\K$.
Then there exists a sharp tridiagonal system $\Phi$ over $\K$ with parameter 
array \eqref{eq:parray} if and only if {\rm (i)}--{\rm (iii)} hold below.
\begin{itemize}
\item[\rm (i)]
$\th_i \neq \th_j$, $\th^*_i \neq \th^*_j$ if $i \neq j$ $(0 \leq i,j \leq d)$.
\item[\rm (ii)]
$\zeta_0=1$, $\zeta_d \neq 0$, and
\begin{equation}  \label{eq:ineq}
 \sum_{i=0}^d \eta_{d-i}(\th_0)\eta^*_{d-i}(\th^*_0) \zeta_i \neq 0.
\end{equation}
\item[\rm (iii)]
The expressions
\begin{equation}              \label{eq:indep}
  \frac{\th_{i-2}-\th_{i+1}}{\th_{i-1}-\th_i},  \qquad\qquad
  \frac{\th^*_{i-2}-\th^*_{i+1}}{\th^*_{i-1}-\th^*_i}
\end{equation}
are equal and independent of $i$ for $2 \leq i \leq d-1$.
\end{itemize}
Suppose {\rm (i)}--{\rm (iii)} hold. Then $\Phi$ is unique up to isomorphism of
tridiagonal systems.
\end{conjecture}

\medskip

In \cite[Section~8]{nomsharp} we proved the ``only if'' direction of
Conjecture \ref{conj:main}.
In \cite[Theorem~1.6]{nomstructure} we proved the last assertion of 
Conjecture \ref{conj:main}.
In this paper we consider what is involved in proving the rest of 
Conjecture \ref{conj:main}.
We are going to define a certain $\K$-algebra $T$ by generators and relations,
and consider the $\K$-algebra $e^*_0Te^*_0$ for a certain idempotent 
$e^*_0 \in T$.
We will state a conjecture about $e^*_0Te^*_0$ called the $\mu$-conjecture.
The $\mu$-conjecture asserts, roughly speaking, that  $e^*_0Te^*_0$ is 
isomorphic to the algebra of all polynomials over $\K$ involving $d$ mutually 
commuting indeterminates.
In Section 10 we show that the $\mu$-conjecture implies Conjecture 
\ref{conj:main}.
In Section 12 we show that the $\mu$-conjecture holds for $d\leq 5$.

\section{The algebra $T$}

\indent
In this section we recall the algebra $T$ from \cite{nomstructure}.
From now until the end of Section 6 let $d$ denote a nonnegative integer and let
$(\{\th_i\}_{i=0}^d; \{\th^*_i\}_{i=0}^d)$
denote a sequence of scalars taken from $\K$ that satisfy conditions
{\rm (i)} and {\rm (iii)} of Conjecture \ref{conj:main}. 

\medskip

The following algebra is reminiscent of an algebra
introduced by E. Egge \cite[Definition 4.1]{Egge}.

\medskip

\begin{definition}{\rm \cite[Definition 2.4]{nomstructure}}\label{def:T}
\samepage
Let $T$ denote the associative $ \K$-algebra with $1$, defined by generators
$a$, $\{e_i\}_{i=0}^d$, $a^*$, $\{e^*_i\}_{i=0}^d$ and relations
\begin{equation}                            \label{eq:eiej}
  e_ie_j=\delta_{i,j}e_i, \qquad 
  e^*_ie^*_j=\delta_{i,j}e^*_i \qquad\qquad 
  (0 \leq i,j \leq d),
\end{equation}
\begin{equation}                            \label{eq:sumei}
  \sum_{i=0}^d e_i=1, \qquad\qquad
  \sum_{i=0}^d e^*_i=1,
\end{equation}
\begin{equation}                                  \label{eq:sumthiei}
   a = \sum_{i=0}^d \theta_ie_i, \qquad\qquad
   a^* = \sum_{i=0}^d \theta^*_i e^*_i,
\end{equation}
\begin{equation}                                     \label{eq:esiakesj}
 e^*_i a^k e^*_j = 0  \qquad \text{if $\;k<|i-j|$}
                  \qquad\qquad (0 \leq i,j,k \leq d),
\end{equation}
\begin{equation}                                      \label{eq:eiaskej}
 e_i {a^*}^k e_j = 0  \qquad \text{if $\;k<|i-j|$} 
                  \qquad\qquad (0 \leq i,j,k \leq d). 
\end{equation}
Let $D$ (resp. $D^*$) denote the $ \K$-subalgebra
of $T$ generated by $a$ (resp. $a^*$).
\end{definition}

\medskip

We now give bases for the $\K$-vector spaces  $D$ and $D^*$.

\medskip

\begin{lemma}  \label{lem:injection}   \samepage
With reference to Definition {\rm \ref{def:T}} the following {\rm (i)}, {\rm (ii)} 
hold.
\begin{itemize}
\item[\rm (i)]
Each of $\{a^i\}_{i=0}^d$, $\{e_i\}_{i=0}^d$ is a basis for $D$.
\item[\rm (ii)]
Each of $\{{a^*}^i\}_{i=0}^d$, $\{e^*_i\}_{i=0}^d$ is a basis for $D^*$.
\end{itemize}
\end{lemma}

\begin{proof}
(i): Observe that for  $0 \leq r,s\leq d$ there exists an $\K$-algebra 
homomorphism $T \to \K$ that sends $e_i \mapsto \delta_{i,r}$ and  
$e^*_i \mapsto \delta_{i,s}$ for $0 \leq i \leq d$.
This is verified by checking that the defining relations for $T$ are respected.
By the observation,  $e_i\not=0$ for $0 \leq i\leq d$.
By this and the equation on the left in (\ref{eq:eiej}),
the elements $\{e_i\}_{i=0}^d$ are linearly independent.
Let $D'$ denote the $\K$-subspace of $T$ spanned by $\{e_i\}_{i=0}^d$.
By the equations on the left in \eqref{eq:eiej}, \eqref{eq:sumei}
we find $D'$ is an $\K$-subalgebra of $T$.
By the equation on the left in \eqref{eq:sumthiei}  and since 
$\{\th_i\}_{i=0}^d$ are mutually distinct, $a$ generates $D'$ so $D=D'$. 
The result follows.

(ii): Similar to the proof of (i) above.
\end{proof}
 
\begin{lemma}    \samepage
With reference to Definition {\rm \ref{def:T}}, 
\begin{equation}         \label{eq:aei}
 ae_i=\theta_ie_i, \qquad\qquad
 a^*e^*_i=\theta^*_ie^*_i    \qquad\qquad (0 \leq i \leq d),
\end{equation}
\begin{equation}         \label{eq:eiesi}
 e_i=\prod_{\stackrel{0\leq j\leq d}{j\neq i}}
      \frac{a-\th_j1}{\th_i-\th_j},             \qquad\qquad
 e^*_i=\prod_{\stackrel{0\leq j\leq d}{j\neq i}}
      \frac{a^*-\th^*_j1}{\th^*_i-\th^*_j}
             \qquad\qquad (0 \leq i \leq d),
\end{equation}
\begin{equation}         \label{eq:prodathi}
 \prod_{i=0}^d (a-\th_i1)=0,    \qquad\qquad
  \prod_{i=0}^d (a^*-\th^*_i1)=0.
\end{equation}
\end{lemma}

\begin{proof}
Routinely verified using \eqref{eq:eiej}--\eqref{eq:sumthiei}.
\end{proof}

\begin{lemma}   \label{lem:tbasis}   \samepage
With reference to Definitions {\rm \ref{def:tau}} and {\rm \ref{def:T}} 
the following {\rm (i), (ii)} hold.
\begin{itemize}
\item[\rm (i)] 
The sequence $\{\tau_i(a)\}_{i=0}^d $ is a basis for $D$.
\item[\rm (ii)] 
The sequence $\{\tau^*_i(a^*)\}_{i=0}^d $ is a basis for $D^*$.
\end{itemize}
\end{lemma}

\begin{proof}
(i): The sequence $\{a^i\}_{i=0}^d$ is a basis for $D$, and the polynomial 
$\tau_i$ has degree exactly $i$ for $0 \leq i\leq d$. The result follows.

(ii): Similar to the proof of (i) above.
\end{proof}

\begin{note}  \label{note:explain}
When we introduced $T$ in \cite{nomstructure} we assumed that there exists 
a tridiagonal system with eigenvalue sequence $\{\th_i\}_{i=0}^d$ and
dual eigenvalue sequence $\{\th^*_i\}_{i=0}^d$.
The assumption was natural in the context of \cite{nomstructure} but it was 
not used in any substantial way.
Indeed one can check that every proof in \cite[Sections 4, 5]{nomstructure} 
is valid verbatim under our present assumption that $\{\th_i\}_{i=0}^d$ and
$\{\th^*_i\}_{i=0}^d$ satisfy conditions (i), (iii) of 
Conjecture \ref{conj:main}.
With this understanding, later in the paper we will invoke some results 
from \cite[Sections 4, 5]{nomstructure}.
\end{note}

\section{Finite-dimensional $T$-modules}

\indent 
In this section we collect some useful facts about finite-dimensional 
$T$-modules.
For the most part the proofs are routine and omitted.

\medskip

\begin{lemma}      \label{lem:basic1}    \samepage
Let $V$ denote a finite-dimensional  $T$-module. 
Then {\rm (i)}--{\rm (v)} hold below.
\begin{itemize}
\item[\rm (i)] 
$V$ is a direct sum of the nonzero spaces among $e_iV$ $(0 \leq i \leq d)$. 
\item[\rm (ii)]
For all $i$ $(0 \leq i \leq d)$ such that $e_iV\not=0$, the space $e_iV$ 
is an eigenspace for $a$ with eigenvalue $\th_i$, and $e_i$ acts on $V$ as 
the projection onto $e_iV$.
\item[\rm (iii)] 
$V$ is a direct sum of the nonzero spaces among $e^*_iV$ $(0 \leq i \leq d)$. 
\item[\rm (iv)]
For all $i$ $(0 \leq i \leq d)$ such that $e^*_iV\not=0$,
the space $e^*_iV$ is an eigenspace for $a^*$ with eigenvalue $\th^*_i$, 
and $e^*_i$ acts on $V$ as the projection onto $e^*_iV$.
\item[\rm (v)] 
Each of $a$, $a^*$ is diagonalizable on $V$.
\end{itemize}
\end{lemma}

\begin{lemma}     \label{lem:basic2}    \samepage
Let $V$ denote a finite-dimensional  $T$-module. 
Then {\rm (i)}, {\rm (ii)} hold below.
\begin{itemize}
\item[\rm (i)] 
For $0 \leq i \leq d$,
\[
a^* e_iV \subseteq e_{i-1}V + e_iV + e_{i+1}V,
\]
where $e_{-1}=0$ and $e_{d+1}=0$.
\item[\rm (ii)]
For $0 \leq i \leq d$,
\[
a e^*_iV \subseteq e^*_{i-1}V + e^*_iV + e^*_{i+1}V,
\]
where $e^*_{-1}=0$ and $e^*_{d+1}=0$.
\end{itemize}
\end{lemma}

\medskip

We now consider finite-dimensional irreducible $T$-modules.

\medskip

\begin{lemma}    \label{lem:basic3}   \samepage
Let $V$ denote a finite-dimensional irreducible $T$-module. Then
{\rm (i)}, {\rm (ii)} hold below.
\begin{itemize}
\item[\rm (i)] 
There exist nonnegative integers $r$, $\delta$ $(r+ \delta \leq d)$ such that
\[
 e^*_iV \not=0 \quad \mbox{\rm if and only if} \quad r \leq i \leq r+ \delta
  \qquad \qquad (0 \leq i \leq d).
\]
\item[\rm (ii)] 
There exist nonnegative integers $t$, $\delta^*$ $(t+ \delta^* \leq d)$
such that
\[
e_iV \not=0 \quad \mbox{\rm if and only if} \quad t \leq i \leq t +\delta^*
\qquad \qquad (0 \leq i \leq d).
\]
\end{itemize}
\end{lemma}

\begin{proof}
(i):
By Lemma \ref{lem:basic1}(iii) and since $V\not=0$, there exists an integer $i$
$(0 \leq i \leq d)$ such that $e^*_iV\not=0$.
Define $r=\min\{i \,|\, 0 \leq i \leq d, \; e^*_iV \not=0\}$ and 
$\rho=\max\{i \,|\,0 \leq i \leq d, \; e^*_iV \not=0\}$.
For $r+1 \leq h \leq \rho-1$ we have $e^*_hV \not=0$; otherwise 
$\sum_{i=r}^{h-1} e^*_iV$ is a nonzero $T$-module properly contained in $V$, 
a contradiction to the irreducibility of $V$. The result follows.

(ii): Similar to the proof of (i) above.
\end{proof}

\begin{proposition}  \label{prop:tds}   \samepage
Let $V$ denote a finite-dimensional irreducible $T$-module
and let $\delta$, $\delta^*$, $r$, $t$ denote the corresponding parameters
from Lemma \ref{lem:basic3}. Then $\delta=\delta^*$.
Moreover the sequence
 $(a;\{e_i\}_{i=t}^{t+\delta}; a^*; \{e^*_i\}_{i=r}^{r+\delta})$ acts on $V$ as
a tridiagonal system. 
\end{proposition}

\begin{proof}
Immediate from Lemmas \ref{lem:basic1}--\ref{lem:basic3} and the third 
sentence below Note \ref{note:star}.
\end{proof}

\begin{proposition}      \label{prop:goback}   \samepage
Fix integers $\delta,r,t$ such that
$0 \leq \delta\leq d$ and $0 \leq r,t\leq d-\delta$. Let
$(A; \{E_i\}_{i=0}^\delta; A^*; \{E^*_i\}_{i=0}^\delta)$ denote a tridiagonal 
system over $\K$ that has eigenvalue sequence $\{\th_i\}_{i=t}^{t+\delta}$
and dual eigenvalue sequence $\{\th^*_i\}_{i=r}^{r+\delta}$.
Let $V$ denote the underlying vector space. 
Then there exists a $T$-module structure on $V$ such that 
{\rm (i)}--{\rm (iii)} hold below.
\begin{itemize}
\item[\rm (i)]
$a$ (resp. $a^*$) acts on $V$ as $A$ (resp. $A^*$).
\item[\rm (ii)]
For $0 \leq i \leq d$, $e_i$ acts on $V$ as $E_{i-t}$ if 
$t \leq i \leq t+\delta$, and zero otherwise.
\item[\rm (iii)]
For $0 \leq i \leq d$, $e^*_i$ acts on $V$ as
$E^*_{i-r}$ if $r \leq i \leq r+\delta$, and zero otherwise.
\end{itemize}
This $T$-module is irreducible.
\end{proposition}

\section{The $\mu$-conjecture}

\indent 
Observe that $e^*_0Te^*_0$ is an $\K$-algebra with multiplicative identity 
$e^*_0$. 
This section contains a general description of $e^*_0Te^*_0$ followed by 
a conjecture about the precise nature of $e^*_0Te^*_0$. 
We start by recalling \cite[Theorem~2.6]{nomstructure} with the wording 
slightly changed.

\medskip

\begin{lemma} {\rm \cite[Theorem~2.6]{nomstructure}} \label{lem:17}   \samepage
The algebra $e^*_0Te^*_0$ is commutative and generated by
\[
  e^*_0\tau_i(a)e^*_0 \qquad \qquad (1\leq i \leq d).
\]
\end{lemma}

\begin{proof}
Follows from \cite[Theorem~2.6]{nomstructure} in view of Lemma 
\ref{lem:tbasis}(i) and Note \ref{note:explain}.
\end{proof}

\begin{definition}      \label{def:18}   \samepage
Let $\{x_i\}_{i=1}^d$ denote mutually commuting indeterminates.
Let $\K[x_1,\ldots,x_d]$ denote the $\K$-algebra consisting of the polynomials 
in $\{x_i\}_{i=1}^d$ that have all coefficients in $\K$. We abbreviate
$R =\K[x_1,\ldots,x_d\rbrack$.
\end{definition}

\begin{corollary}     \label{cor:18a}   \samepage
There exists a surjective $\K$-algebra homomorphism 
$\mu: R \to e^*_0Te^*_0$ that sends $x_i \mapsto e^*_0\tau_i(a)e^*_0$
for $1 \leq i \leq d$.
\end{corollary}

\begin{proof}
Immediate from Lemma \ref{lem:17}.
\end{proof}

\begin{conjecture}    \label{conj:mainp}   \samepage
The map $\mu$ from Corollary {\rm \ref{cor:18a}} is an isomorphism.
\end{conjecture}

\medskip

We call Conjecture \ref{conj:mainp} the {\it $\mu$-conjecture}.
In Section \ref{sec:main} we show that the $\mu$-conjecture implies 
Conjecture \ref{conj:main}.

\medskip

We finish this section with some notation that is motivated by Definition 
\ref{def:split} and Corollary \ref{cor:18a}.

\medskip

\begin{definition}    \label{def:yi}
We define
\begin{equation}        \label{eq:defyi}
 y_i = (\th^*_0-\th^*_1)(\th^*_0-\th^*_2)\cdots(\th^*_0-\th^*_i)x_i
    \qquad\qquad (1 \leq i \leq d).
\end{equation}
\end{definition}

\section{The left ideal $J$ of $T$}

\indent
From now until the end of Section \ref{sec:mainproof} we adopt the following assumption.

\medskip

\begin{assumption} \label{assume}   \samepage
We assume Conjecture \ref{conj:mainp} is true.
Let $d$ denote a nonnegative integer and let
$(\{\th_i\}_{i=0}^d; \{\th^*_i\}_{i=0}^d; \{\zeta_i\}_{i=0}^d)$
denote a sequence of scalars taken from $\K$ that satisfies all three conditions
{\rm (i)}--{\rm (iii)} of  Conjecture \ref{conj:main}. 
Let $T$ denote the $\K$-algebra from Definition \ref{def:T} that is
associated with
the sequence $(\{\th_i\}_{i=0}^d; \{\th^*_i\}_{i=0}^d)$.
\end{assumption}

\medskip

With reference to Assumption \ref{assume}, and with an eye towards
proving Conjecture \ref{conj:main}, we will construct a sharp tridiagonal system
over $\K$ with parameter array 
$(\{\th_i\}_{i=0}^d; \{\th^*_i\}_{i=0}^d$; $\{\zeta_i\}_{i=0}^d)$.
To this end we define a certain left ideal $J$ of $T$ and consider the quotient
$T$-module $M=T/J$.
We will show $M$ is nonzero, finite-dimensional, and has a unique maximal 
proper $T$-submodule $M'$. 
The quotient $T$-module $L=M/M'$ will yield the desired tridiagonal system 
via Proposition \ref{prop:tds}.

\medskip
 
\begin{definition}    \label{def:30}   \samepage
Let $J$ denote the following left ideal of $T$:
\begin{equation}       \label{eq:defj}
J = T(1-e^*_0)+ \sum_{i=1}^d T g_i,
\end{equation}
where
\begin{equation}      \label{eq:gi}
g_i = e^*_0 \tau_i(a)e^*_0 -
 \frac{\zeta_i e^*_0}
      {(\th^*_0-\th^*_1)(\th^*_0-\th^*_2)\cdots(\th^*_0-\th^*_i)}
\qquad \qquad (1 \leq i \leq d).
\end{equation}
\end{definition}

\begin{lemma}    \label{lem:34b}   \samepage
We have
\begin{equation}        \label{eq:tg}
Te^*_0\cap J  = \sum_{i=1}^d T g_i.
\end{equation}
\end{lemma}

\begin{proof}
By \eqref{eq:gi} we have $g_i \in Te^*_0$ for $1 \leq i \leq d$.
Therefore the right-hand side of \eqref{eq:tg} is contained in $Te^*_0$.
The result follows from this, line \eqref{eq:defj},
and since $T = T(1-e^*_0)+Te^*_0$ (direct sum).
\end{proof}

\begin{proposition}    \label{ex:40}   \samepage
We have
\begin{equation}     \label{eq:otods}
e^*_0Te^*_0= \K e^*_0 + e^*_0Te^*_0 \cap J
 \qquad \qquad (\text{\rm direct sum}).
\end{equation}
\end{proposition}

\begin{proof}
We claim   
\begin{equation}    \label{eq:inj}
e^*_0Te^*_0\cap J  = \sum_{i=1}^d e^*_0Te^*_0 g_i.
\end{equation}
To obtain \eqref{eq:inj}, observe that $e^*_0Te^*_0$ contains $g_i$ for 
$1 \leq i \leq d$, so $e^*_0Te^*_0$ contains $\sum_{i=1}^d e^*_0Te^*_0 g_i$.
By Definition \ref{def:30} the ideal $J$ contains 
$\sum_{i=1}^d e^*_0Te^*_0 g_i$ so 
$e^*_0Te^*_0 \cap J$ contains $\sum_{i=1}^d e^*_0Te^*_0 g_i$.
To obtain the reverse inclusion in \eqref{eq:inj}, we fix 
$x \in e^*_0Te^*_0 \cap J$ and show $x \in \sum_{i=1}^d e^*_0Te^*_0 g_i$.
Since $x \in e^*_0Te^*_0$ we have $e^*_0x=x$.
By Lemma \ref{lem:34b} and since $e^*_0Te^*_0 \subseteq Te^*_0$ there exist 
$t_i \in T$ $(1 \leq i \leq d)$ such that $x=\sum_{i=1}^d t_i g_i$. 
In this equation we multiply each term on the left by $e^*_0$ and use
$e^*_0g_i=g_i$ to get $x= \sum_{i=1}^d e^*_0t_ie^*_0 g_i$.
Therefore  $x \in \sum_{i=1}^d e^*_0Te^*_0 g_i$.
We have proved \eqref{eq:inj}.
Now we can easily show \eqref{eq:otods}.
By Corollary \ref{cor:18a}, Definition \ref{def:yi}, and line \eqref{eq:gi},
the map $\mu$ satisfies
\begin{equation}    \label{eq:zz}
 \mu(y_i - \zeta_i)
 = (\th^*_0-\th^*_1)(\th^*_0-\th^*_2)\cdots(\th^*_0-\th^*_i)g_i
    \qquad\qquad    (1 \leq i \leq d).
\end{equation}
Let ${\mathcal J}$ denote the ideal of $R$ generated by 
$\{y_i-\zeta_i\}_{i=1}^d$, so that 
${\cal J} = \sum_{i=1}^d R(y_i -\zeta_i)$.
By Definition \ref{def:18} and \eqref{eq:defyi}
we obtain a direct sum of $\K$-vector spaces
$R=\K 1 + {\mathcal J}$. 
In this equation we apply the isomorphism $\mu$ to each term.
The $\mu$-image of $R$ (resp. $\K 1$) is $e^*_0Te^*_0$ (resp. $\K e^*_0$).
By \eqref{eq:inj}, \eqref{eq:zz} the $\mu$-image of ${\mathcal J}$ is 
$e^*_0Te^*_0\cap J$. 
Line \eqref{eq:otods} follows.
\end{proof}

\section{The $T$-module $M$}

\begin{definition}    \label{def:49}   \samepage
Let $J$ denote the left ideal of $T$ from Definition \ref{def:30}.
Observe that $T/J$ has a natural $T$-module structure;
we abbreviate this $T$-module by $M$.
We abbreviate $\xi$ for the element $1+J$ of $M$.
We note that 
\begin{equation}    \label{eq:vmeaning}
J = \lbrace t \in T\,|\,t\xi=0\rbrace.
\end{equation}
\end{definition}

\begin{lemma}    \label{lem:100}   \samepage
The following {\rm (i), (ii)} hold.
\begin{itemize}
\item[\rm (i)]
$M=T\xi$.
\item[\rm (ii)]
$\xi$ is a basis for $e^*_0M$. 
\end{itemize}
\end{lemma}

\begin{proof}
(i): Recall $M=T/J$ and $\xi=1+J$.

(ii): 
Observe $J\not=T$ by \eqref{eq:otods} so $1 \not\in J$. 
Now $\xi \not=0$ in view of Definition \ref{def:49}.
Since $1-e^*_0 \in J$ we have $(1-e^*_0)\xi=0$, so $\xi = e^*_0\xi$.
Using this and $M=T\xi$ we find $e^*_0M=e^*_0Te^*_0\xi$.
In this equation we evaluate $e^*_0Te^*_0$ using \eqref{eq:otods}, 
\eqref{eq:vmeaning} to get $e^*_0M=\K \xi$. 
The result follows.
\end{proof}

\medskip

Our next goal is to show that the $\K$-vector space $M$ has finite dimension.
We will use the following notation.
For subsets $X, Y$ of $T$ let $XY$ denote the $\K$-subspace of $T$ spanned by
$\{xy \,|\, x \in X,\; y \in Y\}$.

\medskip

\begin{lemma}{\rm \cite[Corollary~4.5]{nomstructure}}\label{cor:EsrDDsDDss}
\samepage
For $0 \leq r,s \leq d$ we have
\[
  e^*_r D D^* D e^*_s =
    \sum_{k=0}^{\lfloor (r+s)/2 \rfloor} 
        e^*_r D e^*_k D e^*_s,
\]
where $\lfloor x \rfloor$ denotes the greatest integer less than or equal to $x$.
\end{lemma}

\begin{definition}   \samepage
Fix an integer $m \geq 1$.
A sequence of integers $(k_0, k_1, \ldots, k_m)$ is called {\it convex} 
whenever $k_{i-1}-k_i \geq k_i - k_{i+1}$ for $1 \leq i \leq m-1$.
\end{definition}

\begin{lemma}    \label{lem:convex}   \samepage
For  $0 \leq r,s\leq d$ and $n\geq 0$ the space
\begin{equation}   \label{eq:Es0DDsDEs0}
  e^*_r D D^* D D^* D \cdots D D^* D e^*_s
    \qquad\qquad (\text{$n+1$ $D$'s})
\end{equation}
is equal to
\begin{equation}   \label{eq:Es0DEs0DEs0}
  \sum e^*_r D e^*_{k_1} D e^*_{k_2} D \cdots D e^*_{k_n} D e^*_s,
\end{equation}
where the sum is over all sequences $(k_1, k_2, \ldots, k_n)$ such that 
$0 \leq k_i \leq  d$ $(1 \leq i \leq n)$ and $(r, k_1, k_2, \ldots, k_n, s)$ 
is convex.
\end{lemma}

\begin{proof}
We assume $n \geq 1$; otherwise there is nothing to prove.
Since $\lbrace e^*_i\rbrace_{i=0}^d$ is a basis for $D^*$
it suffices to show that
\begin{equation}            \label{eq:Es0DEst1DEs0}
   e^*_r D e^*_{k_1} D e^*_{k_2} D \cdots 
             D e^*_{k_n} D e^*_s
\end{equation}
is contained in \eqref{eq:Es0DEs0DEs0} for all sequences
$(k_1, k_2, \ldots, k_n)$
such that $0 \leq k_i  \leq d$ $(1 \leq i \leq n)$.
For each such sequence $(k_1, k_2, \ldots, k_n)$   we define the {\em weight}
to be $\sum_{i=1}^n k_i$.
Suppose there exists a sequence $(k_1, k_2, \ldots, k_n)$ such that 
\eqref{eq:Es0DEst1DEs0} is not contained in \eqref{eq:Es0DEs0DEs0}.
Of all such sequences, pick one with minimal weight. Denote this weight by $w$.
For notational convenience define $k_0=r$ and $k_{n+1}=s$.
The sequence $(k_0, k_1, k_2, \ldots, k_n,k_{n+1})$ is not convex, so there 
exists an integer $i$ $(1 \leq i \leq n)$ such that 
$k_{i-1}-k_i < k_i - k_{i+1}$.
Abbreviate $h = \lfloor (k_{i-1}+k_{i+1})/2 \rfloor$ and note that $h< k_i$.
By Lemma \ref{cor:EsrDDsDDss} the space \eqref{eq:Es0DEst1DEs0} is contained
in the space
\begin{equation}          \label{eq:sumEs0DEst1EsDEs0}
 \sum_{\ell=0}^h
  e^*_r D e^*_{k_1} D \cdots D 
  e^*_{k_{i-1}} D e^*_{\ell} D e^*_{k_{i+1}} \cdots D e^*_{k_n} D e^*_s.
\end{equation}
For $0 \leq \ell \leq h$ the $\ell$-summand in \eqref{eq:sumEs0DEst1EsDEs0} has 
weight less than $w$, so this summand is contained in \eqref{eq:Es0DEs0DEs0}.
Therefore \eqref{eq:Es0DEst1DEs0} is contained in \eqref{eq:Es0DEs0DEs0},
for a contradiction. The result follows.
\end{proof}

\begin{lemma}      \label{lem:mspan}   \samepage
For $1 \leq r \leq d$ the $\K$-vector space $e^*_rM$ is equal to
\begin{equation}   \label{eq:Es0DEs0DEs0m}
  \sum e^*_r D e^*_{k_1} D e^*_{k_2} D \cdots D e^*_{k_m} D \xi,
\end{equation}
where the sum is over all sequences $(k_1, k_2, \ldots, k_m)$ $(m\geq 0)$
such that  $r> k_1 > k_2 > \cdots > k_m > 0$ and 
$(r, k_1, k_2, \ldots, k_m, 0)$ is convex.
\end{lemma}

\begin{proof}
We have $M=T\xi$ and $\xi= e^*_0 \xi$ so $e^*_rM = e^*_rTe^*_0\xi$.
The algebra $T$ is generated by $D,D^*$.
Therefore $e^*_rTe^*_0$  is the sum over $n=0,1,2,\ldots$ of terms 
\eqref{eq:Es0DDsDEs0} (with $s=0$).
We apply these terms to $\xi$ and simplify the result using Lemma
\ref{lem:100}(ii) and Lemma \ref{lem:convex}; this yields terms contained in 
the sum \eqref{eq:Es0DEs0DEs0m}. The result follows.
\end{proof}

\begin{proposition}     \label{prop:mfd}   \samepage
The $\K$-vector space $M$ has finite dimension.
\end{proposition} 

\begin{proof}
We have $M=\sum_{r=0}^d e^*_rM$. The subspace $e^*_0M$ has dimension $1$ by
Lemma \ref{lem:100}(ii). For $1 \leq r \leq d$ the subspace $e^*_rM$ has 
finite dimension by Lemma \ref{lem:mspan}, because in the sum 
\eqref{eq:Es0DEs0DEs0m} there are only finitely many terms and each term has 
finite dimension.
\end{proof}

\begin{lemma}     \label{lem:tauc}   \samepage
For $0 \leq i \leq d$ the following holds on $M$:
\begin{equation}     \label{eq:tauc}
e^*_0\tau_i(a)e^*_0 =
 \frac{\zeta_ie^*_0}
      {(\th^*_0-\th^*_1)(\th^*_0-\th^*_2)\cdots(\th^*_0-\th^*_i)}.
\end{equation}
\end{lemma}

\begin{proof}
First assume $i=0$. Then \eqref{eq:tauc} holds since $\tau_0=1$, 
$e^{*2}_0=e^*_0$, and $\zeta_0=1$.
Next assume $1 \leq i \leq d$.
To  show that \eqref{eq:tauc} holds on $M$ we show $g_iM=0$ where $g_i$ is from
\eqref{eq:gi}.
By \eqref{eq:vmeaning} and since $g_i \in J$ we have $g_i\xi=0$.
By this and Lemma \ref{lem:100}(ii) we find $g_ie^*_0M=0$.
Now $g_iM=0$ since $g_i e^*_0 = g_i$.
\end{proof}

\begin{lemma}    \label{lem:triple}   \samepage
The elements $e^*_0e_0e^*_0$, $e^*_0e_de^*_0$ are nonzero on $M$. 
\end{lemma}

\begin{proof}
By Lemma \ref{lem:100}(ii) $e^*_0$ is nonzero on $M$.
Concerning $e^*_0e_de^*_0$, by the equation on the left in \eqref{eq:eiesi} 
we have  $e_d=\tau_d(a)\tau_d(\theta_d)^{-1}$.
By Lemma \ref{lem:tauc} (with $i=d$)
 $e^*_0\tau_d(a)e^*_0=\eta^*_d(\theta^*_0)^{-1}\zeta_d e^*_0$ on $M$.
Therefore 
$e^*_0e_de^*_0 =\tau_d(\theta_d)^{-1} \eta^*_d(\theta^*_0)^{-1}\zeta_d e^*_0$ 
on $M$. By this and since  $\zeta_d\not=0$ we find $e^*_0e_de^*_0$ is nonzero 
on $M$.
Concerning $e^*_0e_0e^*_0$, by the equation on the left in \eqref{eq:eiesi} 
we have  $e_0=\eta_d(a)\eta_d(\theta_0)^{-1}$.
By \cite[Proposition 5.5]{NT:mu},
$\eta_d = \sum_{i=0}^d \eta_{d-i}(\theta_0)\tau_i $.
By these comments and Lemma \ref{lem:tauc},
\[
 e^*_0e_0e^*_0 = 
  e^*_0 \eta_d(\theta_0)^{-1}\eta^*_d(\theta^*_0)^{-1}
   \sum_{i=0}^d  \eta_{d-i}(\theta_0)\eta^*_{d-i}(\theta^*_0) \zeta_i
\]
on $M$. In the above line the sum is nonzero by \eqref{eq:ineq} so
$e^*_0e_0e^*_0$ is nonzero on $M$.
\end{proof}

\section{The $T$-module $L$}   \label{sec:mainproof}

\indent 
In this section we show that there exists a unique maximal proper $T$-submodule
of $M$. We call this module $M'$ and consider the quotient module 
${L}:=M \slash M'$.   

\medskip

\begin{lemma}    \label{lem:110}   \samepage
Let $V$ denote a proper $T$-submodule of $M$. Then $e^*_0V=0$.
\end{lemma}

\begin{proof}
Suppose $e^*_0 V\not=0$. By construction $e^*_0 V \subseteq e^*_0M$ and 
$e^*_0M$ has basis $\xi$ so $\xi \in V$. The space $V$ is $T$-invariant
and $T\xi=M$ so $M=V$, for a contradiction. We conclude $e^*_0V=0$.
\end{proof}

\begin{lemma}     \label{lem:ww}   \samepage
Let $V$ and $V'$ denote proper $T$-submodules of $M$. Then $V+V'$ is a proper 
$T$-submodule of $M$.
\end{lemma}

\begin{proof}
We show $V+V'\not=M$. Note that $e^*_0(V+V')=e^*_0V+e^*_0V'$.
By Lemma \ref{lem:110} $e^*_0V=0$ and  $e^*_0V'=0$, so $e^*_0(V+V')=0$.
But $e^*_0M\not=0$ by Lemma \ref{lem:100}(ii), so $V+V'\not=M$. 
The result follows.
\end{proof}

\begin{definition}   \samepage
Let $V$ denote a proper $T$-submodule of $M$. Then $V$ is called {\em maximal}
whenever $V$ is not contained in any proper $T$-submodule of $M$, 
besides itself.
\end{definition}

\begin{lemma}     \label{lem:120}
There exists a unique maximal proper $T$-submodule in $M$.
\end{lemma}

\begin{proof}
Concerning existence, consider  
\begin{equation}       \label{eq:sumw}
\sum_{V} V,
\end{equation}
where the sum is over all proper $T$-submodules $V$ of $M$.
The space \eqref{eq:sumw} is a proper $T$-submodule of $M$ by 
Lemma \ref{lem:ww}, and since $M$ has finite dimension.
The $T$-submodule \eqref{eq:sumw} is maximal by the construction.
Concerning uniqueness, suppose $V$ and $V'$ are maximal proper $T$-submodules 
of $M$. By Lemma \ref{lem:ww} $V+V'$ is a proper $T$-submodule of $M$.
The space $V+V'$ contains each of $V$, $V'$, so $V+V'$ is equal to each of
$V$, $V'$ by the maximality of $V$ and $V'$. Therefore $V=V'$ and the result 
follows.
\end{proof}

\begin{definition}   \samepage
Let $M'$ denote the maximal proper $T$-submodule of $M$.
Let $L$ denote the quotient $T$-module $M \slash M'$.
By construction $L$ is nonzero, finite-dimensional,  and irreducible.
\end{definition}

\begin{proposition}           \label{thm:exist}   \samepage
The sequence $(a; \{e_i\}_{i=0}^d;a^*; \{e^*_i\}_{i=0}^d)$ acts on $L$ as 
a sharp tridiagonal system with parameter array
 $(\{\th_i\}_{i=0}^d; \{\th^*_i\}_{i=0}^d; \{\zeta_i\}_{i=0}^d)$.
\end{proposition}

\begin{proof}
We first show that $(a; \{e_i\}_{i=0}^d; a^*; \{e^*_i\}_{i=0}^d)$ acts on
$L$ as a tridiagonal system. This will follow from Proposition \ref{prop:tds}
once we show that the integers $r,t,\delta$ from that proposition are $0,0,d$
respectively.
By construction, for $0 \leq i \leq d$ the dimension of $e_iL $ is equal to 
the dimension of $e_iM$ minus the dimension of $e_iM'$. 
Similarly the dimension of $e^*_iL$ is equal to the dimension of $e^*_iM$ 
minus the dimension of $e^*_iM'$.
Observe $e^*_0M'=0$ by Lemma \ref{lem:110} and since $M'$ is properly 
contained in $M$.
The space $e^*_0M$ has dimension $1$ by Lemma \ref{lem:100}(ii);
therefore $e^*_0L$ has dimension $1$ so  $e^*_0L\not=0$.
Now $r=0$ in view of Lemma \ref{lem:basic3}(i).
Next we show $t=0$. Suppose $t\not=0$. Then  $e_0L=0$ by Lemma 
\ref{lem:basic3}(ii), so $e_0 M \subseteq M'$.
In this containment we apply $e^*_0$ to both sides and use $e^*_0M'=0$ to
get $e^*_0e_0M=0$. 
This implies $e^*_0e_0e^*_0M=0$ which contradicts Lemma \ref{lem:triple}.
Therefore $t=0$. 
Next we show $\delta=d$. Suppose  $\delta\not=d$. Recall $\delta=\delta^*$
by Proposition \ref{prop:tds}, so  $\delta^*\not=d$. Now $e_d L=0$ by
Lemma \ref{lem:basic3}(ii), so $e_d M \subseteq M'$.
In this containment we apply $e^*_0$ to both sides and use $e^*_0 M'=0$ to
get $e^*_0e_d M=0$. This implies  $e^*_0e_de^*_0M=0$ which contradicts
Lemma \ref{lem:triple}. Therefore $\delta=d$. We have shown 
$(r,t,\delta)=(0,0,d)$. Therefore $(a; \{e_i\}_{i=0}^d;a^*; \{e^*_i\}_{i=0}^d)$ acts on $L$ as a tridiagonal system which we denote by $\Phi$.
Observe that $\Phi$ is sharp since $e^*_0 L$  has dimension $1$.
By Lemma \ref{lem:basic1} $\Phi$ has eigenvalue sequence $\{\th_i\}_{i=0}^d$
and dual eigenvalue sequence $\{\th^*_i\}_{i=0}^d$.
By Lemma \ref{lem:tauc} and since the canonical map $M \to L$ is a $T$-module
homomorphism, we have
\[
 e^*_0\tau_i(a)e^*_0 = 
  \frac{\zeta_i e^*_0}
       {(\th^*_0-\th^*_1)(\th^*_0-\th^*_2)\cdots(\th^*_0-\th^*_i)}
     \qquad \qquad (0 \leq i \leq d)
\]
on $L$. By this and Definition \ref{def:split} the sequence 
$\{\zeta_i\}_{i=0}^d$ is the split sequence for $\Phi$.
By these comments $\Phi$ has parameter array
 $(\{\th_i\}_{i=0}^d; \{\th^*_i\}_{i=0}^d; \{\zeta_i\}_{i=0}^d)$
and the result follows.
\end{proof}

\section{The $\mu$-conjecture and the classification conjecture}\label{sec:main}

\indent
In this section we show that the $\mu$-conjecture implies Conjecture 
\ref{conj:main}.
This is our first main result.

\medskip

\begin{theorem}     \label{thm:main}   \samepage
Conjecture {\rm \ref{conj:mainp}} implies Conjecture {\rm \ref{conj:main}}.
\end{theorem}

\begin{proof}
We assume Conjecture \ref{conj:mainp} is true and show Conjecture 
\ref{conj:main} is true.
Let the scalars \eqref{eq:parray} be given.
To prove Conjecture \ref{conj:main} in one direction, assume that there exists
a sharp tridiagonal system $\Phi$ that has parameter array \eqref{eq:parray}.
Then Conjecture \ref{conj:main}(i) holds by the construction,
Conjecture \ref{conj:main}(ii) holds by \cite[Corollary 8.3]{nomsharp},
and Conjecture \ref{conj:main}(iii) holds by \cite[Theorem 11.1]{TD00}.
To prove Conjecture \ref{conj:main} in the other direction,
assume that the scalars \eqref{eq:parray} satisfy Conjecture 
\ref{conj:main}(i)--(iii).
Then by Proposition \ref{thm:exist} there exists a sharp tridiagonal system 
over $\K$ with parameter array \eqref{eq:parray}.
By \cite[Theorem 1.6]{nomstructure} this tridiagonal system is unique up to
isomorphism of tridiagonal systems.
\end{proof}

\section{Tridiagonal pairs over an algebraically closed field}

\indent
In this section we give a variation of Conjecture  \ref{conj:main},
involving tridiagonal systems over an algebraically closed field.
We show that this variation follows from the $\mu$-conjecture.

\medskip

\begin{conjecture}     \label{conj:main2}     \samepage
Assume the field $\K$ is algebraically closed.
Let $d$ denote a nonnegative integer and let
\begin{equation}         \label{eq:parray2}
 (\{\th_i\}_{i=0}^d; \{\th^*_i\}_{i=0}^d; \{\zeta_i\}_{i=0}^d)
\end{equation}
denote a sequence of scalars taken from $\K$.
Then there exists a tridiagonal system $\Phi$ over $\K$ with parameter array 
\eqref{eq:parray2} if and only if {\rm (i)}--{\rm (iii)} hold below.
\begin{itemize}
\item[\rm (i)]
$\th_i \neq \th_j$, $\th^*_i \neq \th^*_j$ if $i \neq j$ $(0 \leq i,j \leq d)$.
\item[\rm (ii)]
$\zeta_0=1$, $\zeta_d \neq 0$, and
\[
 \sum_{i=0}^d \eta_{d-i}(\th_0)\eta^*_{d-i}(\th^*_0) \zeta_i \neq 0.
\]
\item[\rm (iii)]
The expressions
\[
  \frac{\th_{i-2}-\th_{i+1}}{\th_{i-1}-\th_i},  \qquad\qquad
  \frac{\th^*_{i-2}-\th^*_{i+1}}{\th^*_{i-1}-\th^*_i}
\]
are equal and independent of $i$ for $2 \leq i \leq d-1$.
\end{itemize}
Suppose {\rm (i)}--{\rm (iii)} hold. Then $\Phi$ is unique up to isomorphism of
tridiagonal systems.
\end{conjecture}

\begin{theorem}  \label{thm:main2}  \samepage
Conjecture {\rm \ref{conj:mainp}} implies Conjecture {\rm \ref{conj:main2}}.
\end{theorem}

\begin{proof}
By \cite[Theorem 1.3]{nomstructure} every tridiagonal system over an 
algebraically closed field is sharp.
The result follows from this and Theorem \ref{thm:main}.
\end{proof}

\section{The $\mu$-conjecture is true for $d \leq 5$}

\indent
In this section we show that the $\mu$-conjecture holds for $d\leq 5$. 
This is our second main result.

\medskip

\begin{theorem}         \label{thm:mainpd5}
Conjecture {\rm \ref{conj:mainp}} is true for $d \leq 5$.
\end{theorem}

\begin{proof}
Let $d$ be given. Referring to the Appendix, let $V$ denote the $R$-module
consisting of formal $R$-linear combinations of the given basis.
We define a $T$-module structure on $V$ as follows.
An $\K$-linear transformation $\psi:V\to V$ is said to {\em commute} with $R$
whenever $\psi r-r\psi$ is zero on $V$ for all $r \in R$. 
Let $a:V \to V$ and $a^*:V \to V$ denote the unique $\K$-linear
transformations that commute with $R$ and act in the specified way on the basis.
For $0 \leq i \leq d$  define $\K$-linear transformations 
$e_i:V\to V$ and $e^*_i : V\to V$ such that \eqref{eq:eiesi} holds.
By a laborious computation (or with the aid of Mathematica) one can check
that relations \eqref{eq:eiej}--\eqref{eq:eiaskej} hold on $V$.
This gives a $T$-module structure on $V$.
Among the basis elements in the Appendix there is one denoted
$\phi$.
For $1 \leq i \leq d$ we show
\begin{equation}    \label{eq:es0tauiaes0}
  e^*_0\tau_i(a)e^*_0.\phi 
  = \frac{y_i}{(\th^*_0-\th^*_1)(\th^*_0-\th^*_1)\cdots(\th^*_0-\th^*_i)} \phi.
\end{equation}
Assume $d\geq 1$; otherwise there is nothing to show.
From the Appendix $a^*.\phi = \th^*_0 \phi$. 
By this and since $e^*_0 =\eta^*_d(a^*)\eta^*_d(\th^*_0)^{-1}$
we find $e^*_0.\phi = \phi$. 
Among the basis elements in the Appendix, consider the following
elements:
\[
  \phi \qquad r \qquad r^2 \qquad \cdots \qquad r^i \qquad 
  lr^i \qquad l^2r^i \qquad \cdots \qquad l^{i-1}r^i
\]
Abbreviate $r^0=\phi$.
By the data in the Appendix,
$(a-\th_h).r^h = r^{h+1}$ for $0 \leq h \leq i-1$
and
$(a^*-\th^*_{i-h}).l^h r^i = l^{h+1}r^i$  for $0 \leq h \leq i-2$.
Moreover
$(a^*-\th^*_1).l^{i-1}r^i = y_i \phi$.
Therefore 
\[
 (a^*-\th^*_1)(a^*-\th^*_2) \cdots (a^*-\th^*_i)\tau_i(a).\phi = y_i\phi.
\]
In this equation we multiply both sides on the left by $e^*_0$
and simplify the result using the equation on the right in \eqref{eq:aei}.
Evaluating the result further using $\phi = e^*_0.\phi$
yields \eqref{eq:es0tauiaes0}.
By \eqref{eq:defyi}, \eqref{eq:es0tauiaes0}, and Corollary \ref{cor:18a},
\[
  \mu(x_i).\phi=x_i \phi  \qquad\qquad(1 \leq i \leq d).
\]
By this and since $\mu$ is an $\K$-algebra homomorphism,
for $f \in R$ we have $\mu(f).\phi = f \phi$. 
By construction $f\phi \not=0$ if $f\not=0$; therefore $\mu$ in injective
and hence an isomorphism. The result follows.
\end{proof}

\begin{corollary}
Conjectures {\rm \ref{conj:main}} and {\rm \ref{conj:main2}}
are true for $d \leq 5$.
\end{corollary}

\begin{proof}
Follows from Theorems \ref{thm:main}, \ref{thm:main2}, and \ref{thm:mainpd5}.
\end{proof}

\section{Suggestions for future research}

\indent
In this section we give some suggestions for future research.

\medskip

In what follows let $(\{\th_i\}_{i=0}^d; \{\th^*_i\}_{i=0}^d)$
denote a sequence of scalars taken from $\K$ that satisfy the conditions
{\rm (i)}, {\rm (iii)} of Conjecture \ref{conj:main}. 
Let $T$ denote the corresponding algebra from Definition \ref{def:T}.
We are going to describe a subset of $T$ that we think is a basis.
To aid in this description we make a few definitions.

\medskip

\begin{definition}   \samepage
Referring to  Definition \ref{def:T}, we call $\{e_i\}_{i=0}^d$ and
$\{e^*_i\}_{i=0}^d$ the {\em standard generators} for $T$.
We call $\{e^*_i\}_{i=0}^d$ {\em starred} and $\{e_i\}_{i=0}^d$  
{\em nonstarred}.
A pair of standard generators is {\em alternating} whenever one of them is 
starred and the other is nonstarred.
For $0 \leq i \leq d$ we call $i$ the {\em index} of $e_i$ and $e^*_i$.
\end{definition}

\begin{definition}   \samepage
For an integer $n\geq 0$, by a {\it word of length $n$} in $T$ we mean a 
product $u_1u_2 \cdots u_n$ such that $\{u_i\}_{i=1}^n$ are standard 
generators and $u_{i-1}, u_i$ are alternating for $2 \leq i \leq n$.
We interpret the word of length $0$ as the identity element
of $T$. We call this word {\em trivial}.
\end{definition}

\begin{definition}   \samepage
For $0 \leq i,j,r\leq d$ we say $r$ is {\em between} the ordered pair $i,j$ 
whenever $i\geq r>j$ or $i\leq r<j$.
For standard generators $u,v,w$ we say $u$ is {\em between} the ordered pair 
$v,w$ whenever the index of $u$ is between the indices of $v,w$.
\end{definition}

\begin{definition}   \samepage
A word $u_1u_2 \cdots u_n$  in $T$ is called {\em zigzag} (or {\em ZZ} ) 
whenever {\rm (i)}, {\rm (ii)} hold below:
\begin{itemize}
\item[\rm (i)]
$u_i$ is not between $u_{i-1},u_{i+1}$ for $2 \leq i \leq n-1$; 
\item[\rm (ii)]
At least one of $u_{i-1}, u_i$ is not between $u_{i-2}, u_{i+1}$ for $3 \leq i \leq n-1$.
\end{itemize}
\end{definition}

\begin{conjecture}   \samepage
Fix integers $r,s$ $(0 \leq r,s\leq d)$. Then the $\K$-vector space $T$ has a 
basis consisting of the $ZZ$ words that do not involve $e_r,e^*_s$.
\end{conjecture}

\begin{conjecture}   \samepage
The $\K$-vector space $Te^*_0$ has a basis consisting of the nontrivial $ZZ$ 
words that end in $e^*_0$ and do not involve $e_0,e^*_d$.
\end{conjecture}

\medskip

Before we state the next conjecture, we have some comments concerning $Te^*_0$
and the algebra $R$ from Definition \ref{def:18}.
Note that $Te^*_0$ has a (right) $R$-module structure such that $v.r=v \mu(r)$ 
for all $v \in Te^*_0$ and $r \in R$ (the map $\mu$ is from 
Corollary \ref{cor:18a}).
Let $\mbox{\rm End}(Te^*_0)$ denote the $\K$-algebra consisting of all 
$\K$-linear transformations from $Te^*_0$ to $Te^*_0$.
An element $f \in \text{End}(Te^*_0)$ is said to {\it commute with $R$} whenever
$f(v.r)= f(v).r $ for all $v \in Te^*_0$ and all $r \in R$.
Let $\text{End}_R(Te^*_0)$ denote the subalgebra of $\text{End}(Te^*_0)$ 
consisting of the elements which commute with $R$.
Note that $Te^*_0$ has a  (left) $T$-module structure such that 
$t.v\mapsto tv$ for all $t \in T$ and $v \in Te^*_0$.
Observe  that the action of $T$ on $Te^*_0$ commutes with $R$ and therefore 
induces an $\K$-algebra homomorphism $T \to \text{End}_R(Te^*_0)$.

\medskip

\begin{conjecture}   \samepage
The above  map $T \to \text{\rm End}_R(Te^*_0)$ is an injection.
\end{conjecture}

\begin{definition} \label{def:basis}   \samepage
\rm Let $W$ denote a right $R$-module. By an {\em $R$-basis} for $W$ we mean
a sequence $\{z_i\}_{i=1}^n$ of elements in $W$ such that each element of $W$ 
can be written uniquely as $\sum_{i=1}^n z_i.r_i$ with $r_i \in R$ for 
$1 \leq i \leq n$.
\end{definition}

\begin{definition}  \cite[Section 7.4]{Rot}   \samepage
Let $W$ denote a right $R$-module. Then $W$  is called {\em free} whenever $W$ 
has at least one $R$-basis.
In this case the number of elements in a basis is independent of the $R$-basis.
This number is called the {\em rank} of $W$.
\end{definition}

\medskip

We are going to describe a subset of $Te^*_0$ which we think is an $R$-basis. 
To describe the subset we will use the following notation.

\begin{definition}   \samepage
By a {\it feasible} $ZZ$ word in $T$ we mean a nontrivial $ZZ$ word that ends 
in $e^*_0$ and whose indices are mutually distinct.
\end{definition}

\begin{example}   \samepage
For $d\leq 4$ we list the feasible $ZZ$ words in $T$.

\begin{center}
\begin{tabular}{c|c}
$d$  &  feasible $ZZ$ words in $T$ \\
\hline
$0$ & $e^*_0$
\\ \hline
$1$ & $e^*_0, \qquad e_1e^*_0$
\\ \hline
$2$ & $e^*_0, \qquad e_1e^*_0, \qquad e_2e^*_0, \qquad e^*_1e_2e^*_0$
\\ \hline
$3$ & $e^*_0, \qquad e_1e^*_0, \qquad e_2e^*_0, \qquad e_3e^*_0$,
\\ & $e^*_1e_2e^*_0, \qquad e^*_1e_3e^*_0, \qquad e^*_2e_3e^*_0, \qquad 
    e_2e^*_1e_3e^*_0$
\\ \hline
$4$ & $e^*_0, \qquad e_1e^*_0, \qquad e_2e^*_0, \qquad e_3e^*_0, \qquad 
    e_4e^*_0$,
\\ 
& $e^*_1e_2e^*_0, \qquad e^*_1e_3e^*_0, \qquad e^*_1e_4e^*_0, \qquad 
  e^*_2e_3e^*_0,  \qquad e^*_2e_4e^*_0, \qquad e^*_3e_4e^*_0$,
\\
& $e_2e^*_1e_3e^*_0, \qquad e_2e^*_1e_4e^*_0, \qquad e_3e^*_1e_4e^*_0,
  \qquad e_3e^*_2e_4e^*_0, \qquad e^*_2e_3e^*_1e_4e^*_0$
\end{tabular}
\end{center}
\end{example}

\begin{conjecture}    \label{lem:25}   \samepage
The $R$-module  $Te^*_0$ is free with rank $2^d$. Moreover this module has an 
$R$-basis consisting of the feasible $ZZ$ words.
\end{conjecture}

\begin{conjecture}   \samepage
For $0 \leq i \leq d$ the $R$-submodules $e^*_iTe^*_0$ and $e_iTe^*_0$ are both
free with rank $d \choose i$.
\end{conjecture}

\begin{problem}   \samepage
An element of $T$ is called {\it central} whenever it commutes with every 
element of $T$.
The {\em center} $Z(T)$ is the $\K$-subalgebra of $T$ consisting of the 
central elements of $T$. Describe $Z(T)$. Find a generating set for $Z(T)$.
Find a basis for the $\K$-vector space $Z(T)$.
\end{problem}

\begin{problem}   \samepage
For $0 \leq i,j\leq d$ write the word  $e^*_0e_ie^*_je_0$ as a linear 
combination of the words $e^*_0e_re^*_0e_0$ $(0 \leq r\leq d)$.
What are the coefficients in this linear combination?
See \cite[Lemma 14.5]{qrac} for a partial answer.
\end{problem}

\newpage

\section{Appendix}

\indent
In this appendix we give some data that is used in the proof of 
Theorem \ref{thm:mainpd5}.
Let $d$ denote a nonnegative integer at most $5$ and
let $(\{\th_i\}_{i=0}^d;\{\th^*_i\}_{i=0}^d)$ denote a sequence of scalars
taken from $\K$ that satisfy the conditions (i), (iii) of Conjecture
\ref{conj:main}. 
For $0 \leq i \leq d-2$ we define scalars
\[
 \ve_i = (\th_{i+1}-\th_{i+2})(\th^*_{i+1}-\th^*_{i+2}) 
        - (\th_i - \th_{i+1})(\th^*_i - \th^*_{i+1}).
\]

\bigskip
\noindent
{\bf The case $d=0$}

\medskip

The basis is $\;\;\phi$.

The action is $\;\;a.\phi = \th_0 \phi$, $\;\;a^*.\phi = \th^*_0 \phi$.

\bigskip\bigskip
\noindent
{\bf The case $d=1$}

\medskip

The basis is $\;\;\phi,\; r$.

The action of $a$ is
\[
    a.\phi = \th_0 \phi + r, \qquad\qquad a.r = \th_1 r.
\]

The action of $a^*$ is
\[
    a^*.\phi = \th^*_0 \phi, \qquad\qquad  a^*.r = \th^*_1 r + \vphi_1 \phi.
\]

\bigskip\bigskip
\noindent
{\bf The case $d=2$}

\medskip

The basis is $\quad \phi, \; r, \; lr^2, \; r^2$.

The action of $a$ is
\begin{align*}
  a.\phi &= \th_0 \phi + r, &   a.r &= \th_1 r + r^2,
\\
  a.lr^2 &= \th_1 lr^2 + (\vphi_1 - \ve_0) r^2, &  a.r^2 &= \th_2 r^2.
\end{align*}

The action of $a^*$ is
\begin{align*}
 a^*.\phi &= \th^*_0 \phi, &  a^*.r &= \th^*_1 r + \vphi_1 \phi,
\\
 a^*.lr^2 &= \th^*_1 lr^2 + \vphi_2 \phi, &  a^*.r^2 &= \th^*_2 r^2 + lr^2.
\end{align*}

\bigskip\bigskip

For $d \geq 3$ we define 
$\beta \in \K$ such that $\be+1$ is the common value of \eqref{eq:indep}.
We remark that $\be+1$ is nonzero; otherwise $\th_0=\th_3$.
For $d\geq 4$ the scalar $\be$ is nonzero; otherwise $\th_0=\th_4$.
For $d=5$ the scalar $\be^2+\be-1$ is nonzero; otherwise $\th_0=\th_5$.

\newpage

\noindent
{\bf The case $d=3$}

\medskip

The basis is
\[
\begin{array}{ccc}
\phi 
\\
r & lr^2 & l^2r^3
\\
r^2 & lr^3 & rl^2r^3
\\
r^3
\end{array}
\]

\bigskip
\addtolength{\doublerulesep}{-.2\doublerulesep}

The action of $a$ is
\\ \\
\noindent
$\qquad
\begin{array}{r|l}
v \quad & \qquad a.v
\\ \hline \hline
\phi & \th_0 \phi + r
\\ \hline
r & \th_1 r + r^2
\\
lr^2 & \th_1 lr^2 
  + (\vphi_1 -\ve_0) r^2
  + (\be+1)^{-1} lr^3
\\
l^2r^3 & \th_1 l^2r^3 + rl^2r^3
\\ \hline
r^2 & \th_2 r^2 + r^3
\\
lr^3 &
   \th_2 lr^3
   + (\vphi_1
     + (\th_0-\th_1)(\th^*_0-\th^*_3)-(\th_0-\th_3)(\th^*_2-\th^*_3)) r^3
\\
rl^2r^3 & \text{ see below}
\\ \hline
r^3 & \th_3 r^3
\end{array}
$

\bigskip\medskip

$a.rl^2r^3$ is the weighted sum involving the following terms and coefficients.
\\ \\
$\qquad
\begin{array}{r|l}
\text{term} & \qquad \text{coefficient}
\\ \hline
rl^2r^3 & \th_2
\\
r^3 & \vphi_2 
     + \vphi_1(\be+2)((\th_0-\th_1)(\th^*_1-\th^*_2)-(\th_1-\th_2)(\th^*_2-\th^*_3))
\\
    & + ((\th_0-\th_1)(\th^*_0-\th^*_3)-(\th_0-\th_3)(\th^*_2-\th^*_3))
\\
    & \qquad \times ((\th_0-\th_1)(\th^*_0-\th^*_2)-(\th_1-\th_2)(\th^*_1-\th^*_3))
\end{array}
$

\bigskip\bigskip\bigskip\bigskip

The action of $a^*$ is
\\ \\
$\qquad
\begin{array}{r|l}
v \quad & \qquad a^*.v
\\ \hline \hline
\phi & \th^*_0 \phi
\\ \hline
r & \th^*_1 r + \vphi_1 \phi
\\
lr^2 & \th^*_1 lr^2 + \vphi_2 \phi
\\
l^2r^3 & \th^*_1 l^2r^3 + \vphi_3 \phi
\\ \hline
r^2 & \th^*_2 r^2 + lr^2
\\
lr^3 & \th^*_2 lr^3 + l^2r^3
\\
rl^2r^3 & 
  \th^*_2 rl^2r^3
  + \vphi_3 (\be+1)^{-1} r
  + (\vphi_1
     + (\th_0-\th_1)(\th^*_0-\th^*_2)-(\th_1-\th_2)(\th^*_1-\th^*_3)) l^2r^3
\\ \hline
r^3 & \th^*_3 r^3 + lr^3
\end{array}
$

\newpage

\noindent
{\bf The case $d=4$}

\medskip

The basis is
\[
\begin{array}{cccccc}
\phi 
\\
r & lr^2 & l^2r^3 & l^3r^4
\\
r^2 & lr^3 & rl^2r^3 & l^2r^4 & rl^3r^4 & lr^2l^3r^4
\\
r^3 & lr^4 & rl^2r^4 & r^2l^3r^4
\\
r^4
\end{array}
\]

\bigskip

The action of $a$ is
\\ \\
$\qquad
\begin{array}{r|l}
v \quad & \qquad a.v
\\ \hline \hline
\phi & \th_{0} \phi + r
\\ \hline
r & \th_{1} r + r^2
\\
lr^2 & 
  \th_1 lr^2
 + (\vphi_1 -\ve_0)r^2
 + (\be+1)^{-1} lr^3
\\
l^2r^3 & \th_1 l^2r^3 + rl^2r^3
\\
l^3r^4 & \th_1 l^3r^4 + rl^3r^4
\\ \hline
r^2 & \th_2 r^2 + r^3
\\
lr^3 & 
 \th_2 lr^3
 + (\vphi_1 
     - \be^{-1}(\be+1)(\ve_0+\ve_1)) r^3
 + \be^{-1}lr^4
\\
rl^2r^3 & \text{\rm see below}
\\
l^2r^4 & \th_2 l^2r^4 + rl^2r^4
\\
rl^3r^4 & \th_2 rl^3r^4 + r^2l^3r^4
\\
lr^2l^3r^4 & \text{\rm see below}
\\ \hline
r^3 & \th_3 r^3 + r^4
\\
lr^4 & \th_3 lr^4 + (\vphi_1 - \be(\be+1)\ve_1) r^4
\\
rl^2r^4 & \th_3 rl^2r^4 
   + (\vphi_2 - \vphi_1(\be+1)(\be+2)\ve_1 
        + (\be+1)^2\ve_1(\ve_0+(\be+1)\ve_1)) r^4
\\
r^2l^3r^4 & \text{ see below}
\\ \hline
r^4 & \th_4 r^4
\end{array}
$

\bigskip\medskip

$a.rl^2r^3$ is the weighted sum involving the following terms and coefficients.
\\ \\
$\qquad
\begin{array}{r|l}
 \text{\rm term} & \qquad \text{\rm coefficient}
\\ \hline
 rl^2r^3 & \th_2
\\
 r^3 & \vphi_2
        - \vphi_1 \beta^{-1}(\be+2)(\ve_0+\ve_1)
        + \be^{-2}(\be+1)(\ve_0+\ve_1)((\be+1)\ve_0+\ve_1)
\\
 lr^4 & \be^{-2}(-\ve_0+(\be^2+\be-1)\ve_1)
\\
  rl^2r^4 & \be^{-1}
\end{array}
$

\newpage

$a.lr^2l^3r^4$ is the weighted sum involving the following terms and coefficients.
\\ \\
$\qquad
\begin{array}{r|l}
 \text{\rm term} & \qquad \text{\rm coefficient}
\\ \hline
 lr^2l^3r^4 & \th_2
\\
 r^3 &  \vphi_4 \be^{-1}(\be+1)^{-1}
\\
 lr^4 &  \vphi_3 (\be+1)^{-1}  
     - \vphi_2 (\be+1)\ve_1
     + \vphi_1 \be^{-1}(\be+1)^2\ve_1(\ve_0+(\be+1)\ve_1)   
\\
    & - \be^{-1}(\be+1)^2\ve_1(\ve_0+\ve_1)(\ve_0+(\be+1)\ve_1)
\\
 r^2l^3r^4 & \vphi_1 - \be^{-1}((\be+1)\ve_0+(2\be+1)\ve_1)
\end{array}
$

\bigskip\medskip

$a.r^2l^3r^4$ is the weighted sum involving the following terms and coefficients.
\\ \\
$\qquad
\begin{array}{r|l}
 \text{\rm term} & \qquad \text{\rm coefficient}
\\ \hline
r^2l^3r^4 & \th_3
\\
r^4 & 
    \vphi_3
     - \vphi_2(\be+1)^2\ve_1 
     + \vphi_1 \be^{-1}(\be+1)^3\ve_1(\ve_0+(\be+1)\ve_1)
\\
    & - \be^{-1}(\be+1)^3\ve_1(\ve_0+\ve_1)(\ve_0+(\be+1)\ve_1)
\end{array}
$

\bigskip\bigskip\bigskip\bigskip

The action of $a^*$ is
\\ \\
$\qquad
\begin{array}{r|l}
v \quad & \qquad a^*.v
\\ \hline \hline
\phi & \th^*_0 \phi
\\ \hline
r & \th^*_1 r + \vphi_1 \phi
\\
lr^2 & \th^*_1 lr^2 + \vphi_2 \phi
\\
l^2r^3 & \th^*_1 l^2r^3 + \vphi_3 \phi
\\
l^3r^4 & \th^*_1 l^3r^4 + \vphi_4 \phi
\\ \hline
r^2 & \th^*_2 r^2 + lr^2
\\
lr^3 & \th^*_2 lr^3 + l^2r^3
\\
rl^2r^3 & \th^*_2 rl^2r^3
 + \vphi_3 (\be+1)^{-1} r
 + (\vphi_1
    - \be^{-1}((\be+1)\ve_0+\ve_1)) l^2r^3
 + \be^{-1}(\be+1)^{-1} l^3r^4
\\
l^2r^4 & \th^*_2 l^2r^4 + l^3r^4
\\
rl^3r^4 & 
 \th^*_2 rl^3r^4 
 + \vphi_4 \be^{-1} r
 + (\vphi_1
     - \be^{-1}(\be+1)(\ve_0+\ve_1)) l^3r^4
\\
lr^2l^3r^4 & \text{ see below}
\\ \hline
r^3 & \th^*_3 r^3 + lr^3
\\
lr^4 & \th^*_3 lr^4 + l^2r^4
\\
rl^2r^4 &  \th^*_3 rl^2r^4
 + (\vphi_1 - (\ve_0+(\be+2)\ve_1)) l^2r^4
 + (\be+1)^{-1} rl^3r^4
\\
r^2l^3r^4 & \th^*_3 r^2l^3r^4 + lr^2l^3r^4
\\ \hline
r^4 & \th^*_4 r^4 + lr^4
\end{array}
$

\newpage

$a^*.lr^2l^3r^4$ is the weighted sum involving the following terms and coefficients.
\\ \\
$\qquad
\begin{array}{r|l}
 \text{\rm term} & \qquad \text{\rm coefficient}
\\ \hline
 lr^2l^3r^4  & \th^*_2
\\
 r & - \vphi_4 \be^{-2}(\ve_0+(\beta+1)\ve_1)
\\
 lr^2 & \vphi_4 \be^{-1}
\\
 l^3r^4 &
   \vphi_2
     - \vphi_1 \be^{-1}(\be+2)(\ve_0+(\be+1)\ve_1)
     + \be^{-2}(\be+1)^2(\ve_0+\ve_1)(\ve_0+(\be+1)\ve_1)
\end{array}
$

\bigskip\bigskip\bigskip\bigskip

\noindent
{\bf The case $d=5$}

\medskip

The basis is
\[
\begin{array}{cccccccccc}
\phi
\\
r & lr^2 & l^2r^3 & l^3r^4 & l^4r^5
\\
r^2 & lr^3 & rl^2r^3 & l^2r^4 & rl^3r^4 & lr^2l^3r^4 & l^3r^5 & rl^4 r^5
& lr^2l^4r^5 & l^2r^3l^4r^5
\\
r^3 & lr^4 & rl^2r^4 & r^2l^3r^4 & l^2r^5 & rl^3r^5 & lr^2l^3r^5 & r^2l^4r^5
& lr^3l^4r^5 & rl^2r^3l^4r^5
\\
r^4 & lr^5 & rl^2r^5 & r^2l^3r^5 & r^3l^4r^5
\\
r^5
\end{array}
\]

\newpage

The action of $a$ is
\\ \\
$\qquad
\begin{array}{r|l}
v \quad & \qquad a.v
\\ \hline \hline
\phi & \th_0 \phi + r
\\ \hline
r & \th_1 r + r^2
\\
lr^2 & \th_1 lr^2
 + (\vphi_1  - \ve_0) r^2
 + (\be+1)^{-1} lr^3
\\
l^2r^3 & \th_1 l^2r^3 + rl^2r^3
\\
l^3r^4 & \th_1 l^3r^4 + rl^3r^4
\\
l^4r^5 & \th_1 l^4r^5 + rl^4r^5
\\ \hline
r^2 & \th_2 r^2 + r^3
\\
lr^3 & \th_2 lr^3 
 + (\vphi_1 - \be^{-1}(\be+1)(\ve_0+\ve_1)) r^3
 + \be^{-1} lr^4
\\
rl^2r^3 & \text{ see below}
\\
l^2r^4 & \th_2 l^2r^4 + rl^2r^4
\\
rl^3r^4 & \th_2 rl^3r^4 + r^2l^3r^4
\\
lr^2l^3r^4 & \text{ see below}
\\
l^3r^5 & \th_2 l^3r^5 + rl^3r^5
\\
rl^4r^5 & \th_2 rl^4r^5 + r^2l^4r^5
\\
lr^2l^4r^5 &   \th_2 lr^2l^4r^5
 + \frac{\vphi_5(\be+2)}{(\be+1)(\be^2+\be-1)} r^3
 + (\vphi_1 - (\ve_0+(\be+2)\ve_1)) r^2l^4r^5
 + (\be+1)^{-1} lr^3l^4r^5
\\
l^2r^3l^4r^5 & \th_2 l^2r^3l^4r^5 + rl^2r^3l^4r^5
\\ \hline
r^3 & \th_3 r^3 + r^4
\\
lr^4 &   \th_3 lr^4
 + (\vphi_1 - \be(\be+1)\ve_1) r^4
 + \frac{\be+1}{\be^2+\be-1} lr^5
\\
rl^2r^4 & \text{ see below}
\\
r^2l^3r^4 & \text{ see below}
\\
l^2r^5 & \th_3 l^2r^5 + rl^2r^5
\\
rl^3r^5 & \th_3 rl^3r^5 + r^2l^3r^5
\\
lr^2l^3r^5 & \text{ see below}
\\
r^2l^4r^5 & \th_3 r^2l^4r^5 + r^3l^4r^5
\\
lr^3l^4r^5 & \text{ see below}
\\
rl^2r^3l^4r^5 & \text{ see below}
\\ \hline
r^4 & \th_4 r^4 + r^5
\\
lr^5 & \th_4 lr^5 + (\vphi_1 - (\be^2+\be-1)(-\ve_0+(\be^2-1)\ve_1)) r^5
\\
rl^2r^5 & \text{ see below}
\\
r^2l^3r^5 & \text{ see below}
\\
r^3l^4r^5 &  \text{ see below}
\\ \hline
r^5 & \th_5 r^5
\end{array}
$

\newpage

$a.rl^2r^3$ is the weighted sum involving the following terms and coefficients.
\\ \\
$\qquad
\begin{array}{r|l}
 \text{\rm term} & \qquad \text{\rm coefficient}
\\ \hline
 rl^2r^3 &  \th_2 
\\
 r^3 & 
     \vphi_2
     - \vphi_1 \be^{-1}(\be+2)(\ve_0+\ve_1)
     + \be^{-2}(\be+1)(\ve_0+\ve_1)((\be+1)\ve_0+\ve_1)) 
\\
 lr^4 & \be^{-2}(-\ve_0+(\be^2+\be-1)\ve_1)
\\
 rl^2r^4 &  \be^{-1}
\\
 l^2r^5 & - \frac{1}{\be(\be^2+\be-1)}
\end{array} 
$

\bigskip\medskip

$a.lr^2l^3r^4$ is the weighted sum involving the following terms and coefficients.
\\ \\
$\qquad
\begin{array}{r|l}
 \text{\rm term} & \qquad \text{\rm coefficient}
\\ \hline
 lr^2l^3r^4 & \th_2
\\
 r^3 &  \vphi_4 \be^{-1}(\be+1)^{-1}
\\
 lr^4 & \vphi_3 (\be+1)^{-1}
     - \vphi_2(\be+1)\ve_1
     + \vphi_1 \be^{-1}(\be+1)^2\ve_1(\ve_0+(\be+1)\ve_1)
\\
   & - \be^{-1}(\be+1)^2\ve_1(\ve_0+\ve_1)(\ve_0+(\be+1)\ve_1) )
\\
r^2l^3r^4 & \vphi_1 - \be^{-1}((\be+1)\ve_0+(2\be+1)\ve_1)
\\
l^2r^5 &  \frac{\vphi_1(\be+1)(-(\be+1)\ve_0+(\be^3+\be^2-2\be-1)\ve_1)}
          {\be(\be^2+\be-1)}
        - \frac{(\be+1)^2\ve_1(-(\be+1)\ve_0+(\be^3+\be^2-2\be-1)\ve_1)}
          {\be^2+\be-1}
\\
 rl^3r^5 & 
           \frac{-(\be+1)\ve_0+(\be^3+\be^2-2\be-1)\ve_1}
                {\be(\be^2+\be-1)}
\\
 lr^2l^3r^5 &   \frac{1}{\be^2+\be-1}
\\
r^2l^4r^5 &     \frac{1}{\be(\be^2+\be-1)}
\end{array}
$

\bigskip\medskip

$a.rl^2r^4$ is the weighted sum involving the following terms and coefficients.
\\ \\
$\qquad
\begin{array}{r|l}
 \text{\rm term} & \qquad \text{\rm coefficient}
\\ \hline
 rl^2r^4 & \th_3
\\
 r^4 & 
   \vphi_2
   - \vphi_1(\be+1)(\be+2)\ve_1
   + (\be+1)^2\ve_1(\ve_0+(\be+1)\ve_1)
\\
 lr^5 &  \frac{(\be+1)(\be+2)(-\ve_0+(\be^2-2)\ve_1)}{\be^2+\be-1} 
\\
 rl^2r^5 &  \frac{\be+2}{\be^2+\be-1}
\end{array}
$

\bigskip\medskip

$a.r^2l^3r^4$ is the weighted sum involving the following terms and coefficients.
\\ \\
$\qquad
\begin{array}{r|l}
 \text{\rm term} & \qquad \text{\rm coefficient}
\\ \hline
 r^2l^3r^4 & \th_3
\\
 r^4 & \vphi_3 - \vphi_2(\be+1)^2\ve_1
   + \vphi_1 \be^{-1}(\be+1)^3\ve_1(\ve_0+(\be+1)\ve_1)
\\
   & - \be^{-1}(\be+1)^3\ve_1(\ve_0+\ve_1)(\ve_0+(\be+1)\ve_1)
\\
 lr^5 & \frac{(\be+1)^3((\be+1)\ve_0^2-(2\be^3+2\be^2-4\be-3)\ve_0\ve_1
              +(\be^2-2)(\be^3+\be^2-2\be-1)\ve_1^2)}
       {\be(\be^2+\be-1)}
\\
 rl^2r^5 & \frac{(\be+1)^2(-(\be+1)\ve_0+(\be^3+\be^2-2\be-1)\ve_1)}
        {\be(\be^2+\be-1)}
\\
 r^2l^3r^5 &    \frac{\be+1}{\be^2+\be-1}
\end{array}
$

\newpage

$a.lr^2l^3r^5$ is the weighted sum involving the following terms and coefficients.
\\ \\
$\qquad
\begin{array}{r|l}
 \text{\rm term} & \qquad \text{\rm coefficient}
\\ \hline
 lr^2l^3r^5 & \th_3  
\\
 lr^5 &  \vphi_3 (\be+1)^{-1}
     - \vphi_2 \be^{-1}(\be+1)^2(-\ve_0+(\be^2-1)\ve_1)   \rule[.7cm]{0cm}{0cm}
\\
  &+ \vphi_1 \be^{-1}(\be+1)^3(-\ve_0+(\be^2-1)\ve_1)(-\ve_0+(\be^2+\be-1)\ve_1)
\\ 
  & - (\be+1)^2(\be^2+\be-1)\ve_1
     (-\ve_0+(\be^2-1)\ve_1)(-\ve_0+(\be^2+\be-1)\ve_1)   \rule[-.5cm]{0cm}{0cm}
\\
 r^2l^3r^5 & 
   \vphi_1
 - \be^{-1}(-(\be+1)\ve_0+(2\be^3+2\be^2-2\be-1)\ve_1)
\\
 r^3l^4r^5 &   \be^{-1}(\be+1)^{-1}
\end{array}
$

\bigskip\medskip

$a.lr^3l^4r^5$ is the weighted sum involving the following terms and coefficients.
\\ \\
$\qquad
\begin{array}{r|l}
 \text{\rm term} & \qquad \text{\rm coefficient}
\\ \hline
 lr^3l^4r^5 & \th_3
\\
 r^4 &  \frac{\vphi_5}{\be(\be^2+\be-1)} 
\\
 lr^5 & \vphi_4 \be^{-1}
           - \vphi_3 (\be+2)(-\ve_0+(\be^2-1)\ve_1)       \rule[.9cm]{0cm}{0cm}
\\
  & + \vphi_2 \be^{-1}(\be+1)^3(-\ve_0+(\be^2-1)\ve_1)
                                (-\ve_0+(\be^2+\be-1)\ve_1)
\\
   &  - \vphi_1 (\be+1)^3(\be+2)\ve_1(-\ve_0+(\be^2-1)\ve_1)
                                        (-\ve_0+(\be^2+\be-1)\ve_1)
\\
   & + \be^{-1}(\be+1)^3(\be^2+\be-1)\ve_1(\ve_0+(\be+1)\ve_1)
\\
   & \qquad \times (-\ve_0+(\be^2-1)\ve_1)(-\ve_0+(\be^2+\be-1)\ve_1)
                                                         \rule[-.5cm]{0cm}{0cm}
\\
 r^3l^4r^5 &
   \vphi_1
    - \be^{-1}(-\ve_0+(\be+1)(\be^2+\be-1)\ve_1)
\end{array}
$

\newpage

$a.rl^2r^3l^4r^5$ is the weighted sum involving the following terms and coefficients.
\\ \\
$\qquad
\begin{array}{r|l}
 \text{\rm term} & \qquad \text{\rm coefficient}
\\ \hline
 rl^2r^3l^4r^5 & \th_3
\\
 r^4 & \frac{\vphi_1\vphi_5(\be+2)}
        {\be(\be^2+\be-1)}
  - \frac{\vphi_5((\be^2+\be-1)\ve_0+(\be+1)(2\be^2+3\be-1)\ve_1)}
        {\be^2(\be^2+\be-1)}
\\
 lr^5 &  \frac{\vphi_5(\be+2)}{\be(\be^2+\be-1)^2}
      +\vphi_4 \be^{-2} (-(\be^2+\be-1)\ve_0+(\be+1)(\be^3-3\be+1)\ve_1)
                                                             \rule[.9cm]{0cm}{0cm}
\\
  &  -\vphi_3 \be^{-1}(\be+2)(-\ve_0+(\be^2-1)\ve_1)
        (-(\be^2+\be-1)\ve_0+(\be+1)(\be^3-3\be+1)\ve_1)
\\
  & + \vphi_2 \be^{-2}(\be+1)^3(-\ve_0+(\be^2-1)\ve_1)
                                       (-\ve_0+(\be^2+\be-1)\ve_1)
\\
  & \qquad \times
     (-(\be^2+\be-1)\ve_0+(\be+1)(\be^3-3\be+1)\ve_1)
\\
  & - \vphi_1 
      \be^{-1}(\be+2)(\be+1)^3\ve_1(-\ve_0+(\be^2-1)\ve_1)
                                        (-\ve_0+(\be^2+\be-1)\ve_1)
\\
  & \qquad \times
     (-(\be^2+\be-1)\ve_0 + (\be+1)(\be^3-3\be+1)\ve_1)
\\
  & + \be^{-2}(\be+1)^3(\be^2+\be-1)\ve_1(\ve_0+(\be+1)\ve_1)
                                          (-\ve_0+(\be^2-1)\ve_1)
\\
  &  \qquad \times
     (-\ve_0+(\be^2+\be-1)\ve_1)(-(\be^2+\be-1)\ve_0+(\be+1)(\be^3-3\be+1)\ve_1)
\\
 rl^2r^5 & 
    \vphi_4 \be^{-1}
    - \vphi_3(\be+2)(-\ve_0+(\be^2-1)\ve_1)               \rule[.7cm]{0cm}{0cm}
\\
   & + \vphi_2 \be^{-1}(\be+1)^3(-\ve_0+(\be^2-1)\ve_1)
                            (-\ve_0+(\be^2+\be-1)\ve_1)   
\\
   & - \vphi_1 (\be+2)(\be+1)^3\ve_1(-\ve_0+(\be^2-1)\ve_1)
                (-\ve_0+(\be^2+\be-1)\ve_1)
\\
   & + \be^{-1}(\be+1)^3(\be^2+\be-1)\ve_1(\ve_0+(\be+1)\ve_1)
\\
   & \qquad \times (-\ve_0+(\be^2-1)\ve_1)(-\ve_0+(\be^2+\be-1)\ve_1)
\\
 r^3l^4r^5 & \vphi_2
      - \vphi_1 \be^{-1}(\be+2)(-\ve_0+(2\be^2+\be-1)\ve_1)  \rule[.7cm]{0cm}{0cm}
\\
  & - \be^{-2}(\be^2+\be-1)\ve_0^2
      + \be^{-2}(\be+1)(2\be^3-\be^2-5\be+2)\ve_0\ve_1
\\
  & + \be^{-2}(\be+1)^2(3\be-1)(\be^2+\be-1) \ve_1^2
\end{array}
$

\bigskip\medskip

$a.rl^2r^5$ is the weighted sum involving the following terms and coefficients.
\\ \\
$\qquad
\begin{array}{r|l}
 \text{\rm term} & \qquad \text{\rm coefficient}
\\ \hline
rl^2r^5 & \th_4
\\
r^5 & \vphi_2 - \vphi_1 (\be+1)(\be+2)(-\ve_0+(\be^2-1)\ve_1)
\\
  & + (\be+1)(\be^2+\be-1)(-\ve_0+(\be^2-1)\ve_1)(-\ve_0+(\be^2+\be-1)\ve_1)
\end{array}
$

\bigskip\medskip

$a.r^2l^3r^5$ is the weighted sum involving the following terms and coefficients.
\\ \\
$\qquad
\begin{array}{r|l}
 \text{\rm term} & \qquad \text{\rm coefficient}
\\ \hline
r^2l^3r^5 & \th_4
\\
r^5 & \vphi_3 - \vphi_2 \be^{-1}(\be+1)^3(-\ve_0+(\be^2-1)\ve_1)
\\
  & + \vphi_1 \be^{-1}(\be+1)^4(-\ve_0+(\be^2-1)\ve_1)
                             (-\ve_0+(\be^2+\be-1)\ve_1)
\\
  & - (\be+1)^3(\be^2+\be-1)\ve_1(-\ve_0+(\be^2-1)\ve_1)
                  (-\ve_0+(\be^2+\be-1)\ve_1)
\end{array}
$

\newpage

$a.r^3l^4r^5$ is the weighted sum involving the following terms and coefficients.
\\ \\
$\qquad
\begin{array}{r|l}
 \text{\rm term} & \qquad \text{\rm coefficient}
\\ \hline
r^3l^4r^5 & \th_4
\\
r^5 & \vphi_4 - \vphi_3 \be(\be+2)(-\ve_0+(\be^2-1)\ve_1)
\\
  & + \vphi_2 (\be+1)^3(-\ve_0+(\be^2-1)\ve_1)(-\ve_0+(\be^2+\be-1)\ve_1)
\\
   & - \vphi_1 \be(\be+1)^3(\be+2)\ve_1(-\ve_0+(\be^2-1)\ve_1)
                                          (-\ve_0+(\be^2+\be-1)\ve_1)
\\
   & + (\be+1)^3(\be^2+\be-1)\ve_1(\ve_0+(\be+1)\ve_1)
       (-\ve_0+(\be^2-1)\ve_1)(-\ve_0+(\be^2+\be-1)\ve_1)
\end{array}
$

\newpage

The action of $a^*$ is
\\ \\
$\qquad
\begin{array}{r|l}
v \quad & \qquad a^*.v
\\ \hline \hline
\phi & \th^*_0 \phi
\\ \hline
r & \th^*_1 r + \vphi_1 \phi
\\
lr^2 & \th^*_1 lr^2 + \vphi_2 \phi
\\
l^2r^3 & \th^*_1 l^2r^3 + \vphi_3 \phi
\\
l^3r^4 & \th^*_1 l^3r^4 + \vphi_4 \phi
\\
l^4r^5 & \th^*_1 l^4r^5 + \vphi_5 \phi
\\ \hline
r^2 & \th^*_2 r^2 + lr^2
\\
lr^3 & \th^*_2 lr^3 + l^2r^3
\\
rl^2r^3 &  \th^*_2 rl^2r^3
 + \vphi_3 (\be+1)^{-1} r
 + (\vphi_1 - \be^{-1}((\be+1)\ve_0+\ve_1)) l^2r^3
 + \be^{-1}(\be+1)^{-1} l^3r^4
\\
l^2r^4 & \th^*_2 l^2r^4 + l^3r^4
\\
rl^3r^4 &  \th^*_2 rl^3r^4
 + \vphi_4 \be^{-1} r
 + (\vphi_1 - \be^{-1}(\be+1)(\ve_0+\ve_1)) l^3r^4
 + \frac{1}{\be(\be^2+\be-1)} l^4r^5
\\
lr^2l^3r^4 & \text{ see below}
\\
l^3r^5 & \th^*_2 l^3r^5 + l^4r^5
\\
rl^4r^5 &  \th^*_2 rl^4r^5
 + \frac{\vphi_5(\be+1)}{\be^2+\be-1} r
 + (\vphi_1 - (\ve_0+(\be+1)\ve_1)) l^4r^5
\\
lr^2l^4r^5 & \text{ see below}
\\
l^2r^3l^4r^5 & \text{ see below}
\\ \hline
r^3 & \th^*_3 r^3 + lr^3
\\
lr^4 & \th^*_3 lr^4 + l^2r^4
\\
rl^2r^4 &  \th^*_3 rl^2r^4
 + (\vphi_1 - (\ve_0+(\be+2)\ve_1)) l^2r^4
 + (\be+1)^{-1} rl^3r^4
 + \frac{\be+2}{(\be+1)(\be^2+\be-1)} l^3r^5
\\
r^2l^3r^4 & \th^*_3 r^2l^3r^4 + lr^2l^3r^4
\\
l^2r^5 & \th^*_3 l^2r^5 + l^3r^5
\\
rl^3r^5 & \th^*_3 rl^3r^5
 + (\vphi_1 - \be^{-1}(-\ve_0+(\be+1)(\be^2+\be-1)\ve_1)) l^3r^5
 + \be^{-1} rl^4r^5
\\
lr^2l^3r^5 & \text{ see below}
\\
r^2l^4r^5 & \th^*_3 r^2l^4r^5 + lr^2l^4r^5
\\
lr^3l^4r^5 & \th^*_3 lr^3l^4r^5 + l^2r^3l^4r^5
\\
rl^2r^3l^4r^5 & \text{ see below}
\\ \hline
r^4 & \th^*_4 r^4 + lr^4
\\
lr^5 & \th^*_4 lr^5 + l^2r^5
\\
rl^2r^5 &   \th^*_4 rl^2r^5 
 + (\vphi_1 - (\be+1)(-\ve_0+(\be^2+\be-2)\ve_1)) l^2r^5
 + (\be+1)^{-1} rl^3r^5
\\
r^2l^3r^5 & \th^*_4 r^2l^3r^5 + lr^2l^3r^5
\\
r^3l^4r^5 & \th^*_4 r^3l^4r^5 + lr^3l^4r^5
\\ \hline
r^5 & \th^*_5 r^5 + lr^5
\end{array}
$

\newpage

$a^*.lr^2l^3r^4$ is the weighted sum involving the following terms and coefficients.
\\ \\
$\qquad
\begin{array}{r|l}
 \text{\rm term} & \text{\rm coefficient}
\\ \hline
 lr^2l^3r^4 & \th^*_2
\\
 r &  \frac{\vphi_5(\be+2)}{\be(\be^2+\be-1)^2}
   - \vphi_4 \be^{-2}(\ve_0+(\be+1)\ve_1)
\\
 lr^2 & \vphi_4 \be^{-1}
\\
 l^3r^4 &
   \vphi_2
   - \vphi_1 \be^{-1}(\be+2)(\ve_0+(\be+1)\ve_1)
   + \be^{-2}(\be+1)^2(\ve_0+\ve_1)(\ve_0+(\be+1)\ve_1)
\\
 l^4r^5 &  \frac{\vphi_1(\be+2)}
       {\be(\be^2+\be-1)}
   - \frac{(\be+1)^2(\ve_0+(\be+1)\ve_1)}
       {\be^2(\be^2+\be-1)}
\end{array}
$

\bigskip\medskip

$a^*.lr^2l^4r^5$ is the weighted sum involving the following terms and coefficients.
\\ \\
$\qquad
\begin{array}{r|l}
 \text{\rm term} & \text{\rm coefficient}
\\ \hline
 lr^2l^4r^5 & \th^*_2 
\\
 r &  - \frac{\vphi_5(\be+1)(\be+2)\ve_1}{\be^2+\be-1}
\\
 lr^2 &  \frac{\vphi_5(\be+2)}{\be^2+\be-1} 
\\
 l^4r^5 & 
   \vphi_2
   - \vphi_1 (\be+1)(\be+2)\ve_1
   + (\be+1)^2\ve_1(\ve_0+(\be+1)\ve_1)
\end{array}
$

\bigskip\medskip

$a^*.l^2r^3l^4r^5$ is the weighted sum involving the following terms and coefficients.
\\ \\
$\qquad
\begin{array}{r|l}
 \text{\rm term} & \text{\rm coefficient}
\\ \hline
 l^2r^3l^4r^5 & \th^*_2
\\
 r &  \frac{\vphi_5(\be+1)^3\ve_1(-\ve_0+(\be^2+\be-1)\ve_1)}
        {\be(\be^2+\be-1)}
\\
 lr^2 &  - \frac{\vphi_5(\be+1)^2(-\ve_0+(\be^2+\be-1)\ve_1)}
            {\be(\be^2+\be-1)}
\\
 l^2r^3 &  \frac{\vphi_5(\be+1)}{\be^2+\be-1} 
\\
 l^4r^5 & \vphi_3 - \vphi_2 \be^{-1}(\be+1)^2(-\ve_0+(\be^2+\be-1)\ve_1)
                                                              \rule[.7cm]{0cm}{0cm}
\\
  & + \vphi_1 \be^{-1}(\be+1)^4\ve_1(-\ve_0+(\be^2+\be-1)\ve_1)
\\
  & - (\be+1)^3\ve_1(\ve_0+(\be+1)\ve_1)(-\ve_0+(\be^2+\be-1)\ve_1)
\end{array}
$

\bigskip\medskip

$a^*.lr^2l^3r^5$ is the weighted sum involving the following terms and coefficients.
\\ \\
$\qquad
\begin{array}{r|l}
 \text{\rm term} & \text{\rm coefficient}
\\ \hline
 lr^2l^3r^5 & \th^*_3
\\
 r^2 &  - \frac{\vphi_5}{\be(\be^2+\be-1)}
\\
 l^3r^5 & \vphi_2
      - \vphi_1 \be^{-1}(\be+2)(-(\be+1)\ve_0+(\be^3+2\be^2-\be-1)\ve_1) 
                                                               \rule[.7cm]{0cm}{0cm}
\\
  & + \be^{-2}(\be+1)\ve_0^2 
     - \be^{-2}(\be^5+3\be^4+5\be^3+3\be^2-3\be-2)\ve_0\ve_1  
\\
  & + \be^{-2}(\be+1)(\be^2+\be-1)(\be^4+2\be^3+\be^2-2\be-1)\ve_1^2
                                                               \rule[-.5cm]{0cm}{0cm}
\\
 rl^4r^5 & - \be^{-2}(-(\be+1)\ve_0+(\be^3+\be^2-2\be-1)\ve_1)
\\
 lr^2l^4r^5 & \be^{-1}
\end{array}
$

\newpage

$a^*.rl^2r^3l^4r^5$ is the weighted sum involving the following terms and coefficients.
\\ \\
$\qquad
\begin{array}{r|l}
 \text{\rm term} & \text{\rm coefficient}
\\ \hline
 rl^2r^3l^4r^5 & \th^*_3
\\
 r^2 & \frac{\vphi_5(\be+1)(\ve_0+(\be+1)\ve_1)(-\ve_0+(\be^2+\be-1)\ve_1)}
          {\be(\be^2+\be-1)}
    - \frac{\vphi_1\vphi_5(\be+1)(-\ve_0+(\be^2+\be-1)\ve_1)}
            {\be(\be^2+\be-1)}
\\
 lr^3 & - \frac{\vphi_5(-\ve_0+(\be^2+\be-1)\ve_1)}
               {\be(\be^2+\be-1)}
\\
 rl^2r^3 & \frac{\vphi_5}{\be^2+\be-1}
\\
 l^2r^4 & \frac{\vphi_5}{\be(\be^2+\be-1)}
\\
 l^3r^5 &  \vphi_4 \be^{-1}(\be+1)^{-1}
        - \vphi_3(\be+1)^{-1}(\be+2)(-\ve_0+(\be^2-1)\ve_1)  \rule[.9cm]{0cm}{0cm}
\\
  & + \vphi_2 \be^{-1}(\be+1)^2(-\ve_0+(\be^2-1)\ve_1)
                               (-\ve_0+(\be^2+\be-1)\ve_1)
\\
  & - \vphi_1 (\be+1)^2(\be+2)\ve_1(-\ve_0+(\be^2-1)\ve_1)
                                   (-\ve_0+(\be^2+\be-1)\ve_1)
\\
  & + \be^{-1}(\be+1)^2(\be^2+\be-1)\ve_1(\ve_0+(\be+1)\ve_1)
\\
  & \qquad \times  (-\ve_0+(\be^2-1)\ve_1)(-\ve_0+(\be^2+\be-1)\ve_1)
\\
 rl^4r^5 & \vphi_3 (\be+1)^{-1}
          - \vphi_2 \be^{-1}(\be+1)(-\ve_0+(\be^2+\be-1)\ve_1)  \rule[.7cm]{0cm}{0cm}
\\
  & + \vphi_1 \be^{-1}(\be+1)^3\ve_1(-\ve_0+(\be^2+\be-1)\ve_1)
\\
  & - (\be+1)^2\ve_1(\ve_0+(\be+1)\ve_1)(-\ve_0+(\be^2+\be-1)\ve_1)
                                                               \rule[-.5cm]{0cm}{0cm}
\\
 l^2r^3l^4r^5 & \vphi_1 - \be^{-1}((\be-1)\ve_0+(2\be^2+2\be-1)\ve_1)
\end{array}
$

\section{Acknowledgment}

The authors thank John Caughman, Brian Curtin,
 Eric Egge, Mark MacLean, and Michael Lang for giving
this paper a close reading and offering many valuable
suggestions.

\newpage

\noindent Kazumasa Nomura \hfil\break
\noindent College of Liberal Arts and Sciences \hfil\break
\noindent Tokyo Medical and Dental University \hfil\break
\noindent Kohnodai, Ichikawa, 272-0827 Japan \hfil\break
\noindent email: {\tt knomura@pop11.odn.ne.jp} \hfil\break


\bigskip

\noindent Paul Terwilliger \hfil\break
\noindent Department of Mathematics \hfil\break
\noindent University of Wisconsin \hfil\break
\noindent 480 Lincoln Drive \hfil\break
\noindent Madison, WI 53706-1388 USA \hfil\break
\noindent email: {\tt terwilli@math.wisc.edu }\hfil\break


{
\small

\begin{thebibliography}{10}

\bibitem{hasan}
H.~Alnajjar, B.~Curtin,
A family of tridiagonal pairs,
Linear Algebra Appl. 390 (2004) 369--384.

\bibitem{hasan2}
H.~Alnajjar,  B.~Curtin.
A family of tridiagonal pairs related to
the quantum affine algebra $U\sb q(\widehat{\mathfrak{sl}}\sb 2)$,
Electron. J. Linear Algebra 13 (2005) 1--9. 

\bibitem{CurtH}
H.~Alnajjar, B.~Curtin,
A bilinear form for tridiagonal pairs of $q$-Serre type,
submitted for publication.

\bibitem{AWil}
R.~Askey, J.A.~Wilson,
A set of orthogonal polynomials that generalize the {R}acah coefficients or $6-j$ symbols. 
SIAM J. Math. Anal. 10 (1979) 1008--1016. 

\bibitem{bas1}
P.~Baseilhac,
Deformed {D}olan-{G}rady relations in quantum integrable models,
Nuclear Phys. B 709 (2005) 491--521.

\bibitem{bas2}
P.~Baseilhac,
An integrable structure related with tridiagonal algebras,
Nuclear Phys. B 705 (2005) 605--619.

\bibitem{bas5}
P.~Baseilhac,
The $q$-deformed analogue of the Onsager algebra: 
beyond the Bethe ansatz approach,
Nuclear Phys. B 754 (2006) 309--328.

\bibitem{bas6}
P.~Baseilhac,
A family of tridiagonal pairs and related symmetric functions,
J. Phys. A 39 (2006) 11773--11791. 

\bibitem{bas3}
P.~Baseilhac, K.~Koizumi,
A new (in)finite dimensional algebra for quantum integrable models,
Nuclear Phys. B 720 (2005) 325--347;
{\tt arXiv:math-ph/0503036}.

\bibitem{bas4}
P.~Baseilhac, K.~Koizumi.
A deformed analogue of Onsager's symmetry in the XXZ open spin chain,
J. Stat. Mech. Theory Exp. 2007, no. 10,  P10005, 15 pp. (electronic);
{\tt arXiv:hep-th/0507053}.

\bibitem{bas7}
P.~Baseilhac, K.~Koizumi,
Exact spectrum of the $XXZ$ open spin chain from the  $q$-Onsager algebra 
representation theory,
J. Stat. Mech. Theory Exp.
2007,  no. 9, P09006, 27 pp. (electronic).

\bibitem{Egge}
E.~Egge,
A generalization of the Terwilliger algebra,
J. Algebra 233 (2000) 213--252.

\bibitem{Ha}
B.~Hartwig,
The tetrahedron algebra and its finite-dimensional irreducible modules,
Linear Algebra Appl. 422 (2007) 219--235;
{\tt arXiv:math.RT/0606197}.

\bibitem{TD00}
T.~Ito, K.~Tanabe, P.~Terwilliger,
Some algebra related to ${P}$- and ${Q}$-polynomial association schemes,  
in: Codes and Association Schemes (Piscataway NJ, 1999), Amer. Math. Soc., 
Providence RI, 2001, pp. 167--192; 
{\tt arXiv:math.CO/0406556}.

\bibitem{shape}
T.~Ito, P.~Terwilliger,
The shape of a tridiagonal pair,
J. Pure Appl. Algebra 188 (2004) 145--160;
{\tt arXiv:math.QA/0304244}.

\bibitem{tdanduq}
T.~Ito, P.~Terwilliger,
Tridiagonal pairs and the quantum affine algebra $U_q({\widehat{sl}}_2)$,
Ramanujan J. 13 (2007) 39--62;
{\tt arXiv:math.QA/0310042}.

\bibitem{NN}
T.~Ito, P.~Terwilliger,
Two non-nilpotent linear transformations that satisfy the cubic $q$-Serre relations.
J. Algebra Appl. 6 (2007) 477--503;
{\tt arXiv:math.QA/0508398}.

\bibitem{qtet}
T.~Ito, P.~Terwilliger,
The $q$-tetrahedron algebra and its finite-dimensional irreducible modules,
Comm. Algebra 35 (2007) 3415--3439; 
{\tt arXiv:math.QA/0602199}.

\bibitem{ITdrg}
T.~Ito, P.~Terwilliger,
Distance-regular graphs and the $q$-tetrahedron algebra,
European J. Combin., in press;
{\tt arXiv:math.CO/0608694}.

\bibitem{Ev}
T.~Ito, P.~Terwilliger,
Finite-dimensional irreducible modules for the three-point $\mathfrak{sl}_2$ loop algebra,
Comm. Algebra, in press;
{\tt arXiv:0707.2313}.

\bibitem{IT:Krawt}
T.~Ito, P.~Terwilliger.
Tridiagonal pairs of Krawtchouk type,
Linear Algebra Appl. 427 (2007) 218--233;
{\tt arXiv:0706.1065}.

\bibitem{IT:aug}
T.~Ito, P.~Terwilliger.
The augmented tridiagonal algebra,
preprint.

\bibitem{KoeSwa}
R.~Koekoek, R.~F.~Swarttouw.
The Askey scheme of hypergeometric orthogonal polyomials and its   $q$-analog, 
report 98-17, Delft University of Technology, The Netherlands, 1998;
Available at
{\tt http://aw.twi.tudelft.nl/{\~{}}koekoek/askey.html}

\bibitem{N:aw}
K.~Nomura,
Tridiagonal pairs and the {A}skey-{W}ilson relations,
Linear Algebra Appl. 397 (2005) 99--106.

\bibitem{N:refine}
K.~Nomura,
A refinement of the split decomposition of a tridiagonal pair,
Linear Algebra Appl. 403 (2005) 1--23.

\bibitem{nom4}
K.~Nomura,
Tridiagonal pairs of height one,
Linear Algebra Appl. 403 (2005) 118--142.

\bibitem{NT:balanced}
K.~Nomura, P.~Terwilliger,
Balanced Leonard pairs,
Linear Algebra Appl. 420 (2007) 51--69;
{\tt arXiv:math.RA/0506219}.

\bibitem{NT:formula}
K.~Nomura, P.~Terwilliger,
Some trace formulae involving the split sequences of a Leonard pair,
Linear Algebra Appl. 413 (2006) 189--201;
{\tt arXiv:math.RA/0508407}.

\bibitem{NT:det}
K.~Nomura, P.~Terwilliger,
The determinant of $AA^*-A^*A$ for a Leonard pair $A,A^*$,
Linear Algebra Appl. 416 (2006) 880--889;
{\tt arXiv:math.RA/0511641}.

\bibitem{NT:mu}
K.~Nomura, P.~Terwilliger,
Matrix units associated with the split basis of a Leonard pair,
Linear Algebra Appl. 418 (2006) 775--787;
{\tt arXiv:math.RA/0602416}.

\bibitem{NT:span}
K.~Nomura, P.~Terwilliger,
Linear transformations that are tridiagonal with respect to
both eigenbases of a Leonard pair,
Linear Algebra Appl. 420 (2007) 198--207;
{\tt arXiv:math.RA/0605316}.

\bibitem{NT:switch}
K.~Nomura, P.~Terwilliger,
The switching element for a Leonard pair,
Linear Algebra Appl. 428 (2008) 1083--1108;
{\tt arXiv:math.RA/0608623}.

\bibitem{nomsplit}
K.~Nomura, P.~Terwilliger,
The split decomposition of a tridiagonal pair,
Linear Algebra Appl. 424 (2007) 339--345;
{\tt arXiv:math.RA/0612460}. 

\bibitem{nomsharp}
K.~Nomura, P.~Terwilliger,
Sharp tridiagonal pairs,
Linear Algebra Appl., in press;
{\tt arXiv:0712.3665}.


\bibitem{nomtowards}
K.~Nomura, P.~Terwilliger,
Towards a classification of the tridiagonal pairs,
Linear Algebra Appl., in press;
{\tt arXiv:0801.0621}.

\bibitem{nomstructure}
K.~Nomura, P.~Terwilliger,
The structure of a tridiagonal pair,
Linear Algebra Appl., submitted for publication;
{\tt arXiv:0802.1096}.

\bibitem{Rot}
J.~J.~Rotman,
Advanced modern algebra,
Prentice Hall,
Saddle River NJ 2002.

\bibitem{LS99}
P.~Terwilliger,
Two linear transformations each tridiagonal with respect to an
eigenbasis of the other,
Linear Algebra Appl. 330 (2001) 149--203;
{\tt arXiv:math.RA/0406555}.

\bibitem{qSerre}
P.~Terwilliger,
Two relations that generalize the $q$-Serre relations and the
Dolan-Grady relations,
Physics and Combinatorics 1999 (Nagoya), World Scientific Publishing,
 River Edge, NJ, 2001, pp. 377--398; 
{\tt arXiv:math.QA/0307016}.

\bibitem{TLT:array}
P.~Terwilliger,
Two linear transformations each tridiagonal with respect to an
eigenbasis of the other; comments on the parameter array.
Des. Codes Cryptogr. 34 (2005) 307--332;
{\tt arXiv:math.RA/0306291}.

\bibitem{qrac}
P.~Terwilliger,
Leonard pairs and the $q$-Racah polynomials,
Linear Algebra Appl. 387 (2004) 235--276;
{\tt arXiv:math.QA/0306301}.

\bibitem{madrid}
P.~Terwilliger,
An algebraic approach to the Askey scheme of orthogonal polynomials,
Orthogonal polynomials and special functions, 
Lecture Notes in Math., 1883, 
Springer, Berlin, 2006, pp. 255--330;
{\tt arXiv:math.QA/0408390}. 

\bibitem{aw}
P.~Terwilliger, R.~Vidunas,
Leonard pairs and the Askey-Wilson relations,
J. Algebra Appl. 3 (2004) 411--426;
{\tt arXiv:math.QA/0305356}.

\bibitem{Vidar}
M.~Vidar,
Tridiagonal pairs of shape $(1,2,1)$,
Linear Algebra Appl., in press;
{\tt arXiv:0802.3165}.
\end{thebibliography}

}
\end{document}